\documentclass[reqno]{amsart}
\usepackage[usenames,dvipsnames]{xcolor}
\usepackage{amssymb,epsfig,graphics,graphicx,bbm,hyperref}
\usepackage{pgf,pgfarrows,pgfnodes,pgfautomata,pgfheaps,pgfshade}
\usepackage{mathrsfs}
\usepackage{array}

\usepackage{multicol}
\usepackage{enumitem}
\usepackage{tikz-cd}
\usetikzlibrary{matrix,arrows,decorations.pathmorphing,positioning}
\usepackage{nicefrac}
\usepackage[displaymath, mathlines]{lineno}

\definecolor{gray}{rgb}{0.93,0.93,0.93}
\definecolor{light-gold}{rgb}{0.99,0.97,0.78}

\definecolor{gold}{rgb}{0.7,0.55,0}
\newcommand\cec{}
\newcommand\jeb{}
\newcommand\peter[1]{#1}

\def\be{\begin{equation}}
\def\ee{\end{equation}}
\def\bm{\begin{multline}}
\def\bfig{\begin{figure}[htb]}
\def\efig{\end{figure}}
\newcommand{\dd}{{\rm d}}

\newcommand{\e}[1]{\,{\rm e}^{#1}\,}

\numberwithin{equation}{section}
\newtheorem{theorem}{Theorem}[section]
\newtheorem{proposition}[theorem]{Proposition}
\newtheorem{lemma}[theorem]{Lemma}
\newtheorem{corollary}[theorem]{Corollary}

\newtheorem{assumption}[theorem]{Assumption}
\newtheorem{example}[theorem]{Example}

\newcommand{\eps}{{\varepsilon}}

\def\ba#1\ea{\begin{linenomath}\begin{align*}#1\end{align*}\end{linenomath}}
\def\ban#1\ean{\begin{linenomath}\begin{align}#1\end{align}\end{linenomath}}

\newcommand{\bs}{\boldsymbol}

\makeatletter
  \def\tagform@#1{\maketag@@@{\scriptsize{(#1)}\@@italiccorr}}
\makeatother

\renewcommand{\eqref}[1]{(\ref{#1})}

\newcommand{\EE}{\mathbb{E}} 

\newcommand{\PP}{\mathbb{P}}

\newcommand{\oo}{\infty}

\begin{document}


\title[A two-table theorem for a disordered
  Chinese restaurant process]{
  {A two-table theorem for a disordered
  Chinese restaurant process}}

\author[Bj\"ornberg]{Jakob E. Bj\"ornberg}
\address{Department of Mathematics,
Chalmers University of Technology and the University of Gothenburg,
Sweden}
\email{jakob.bjornberg@gu.se}
 
\author[Mailler]{C\'ecile Mailler}
\address{Department of Mathematical Sciences, University of Bath, Bath BA2 7AY, U.K.}
\email{c.mailler@bath.ac.uk}

\author[M\"orters]{Peter M\"orters}
\address{Universit\"at zu K\"oln, Department Mathematik, Weyertal 86--90, 50931 K\"oln,
Germany}
\email{moerters@math.uni-koeln.de}

\author[Ueltschi]{Daniel Ueltschi}
\address{Department of Mathematics, University of Warwick,
Coventry, CV4 7AL, U.K.}
\email{daniel@ueltschi.org}



\begin{abstract} 
We investigate a disordered variant of Pitman's
Chinese restaurant process where tables carry i.i.d.\ weights.  
Incoming customers choose to 
sit at an occupied table with a probability proportional to the
product of its occupancy and its weight, {or they sit at an unoccupied table with a probability proportional  to a parameter $\theta>0$.} 
This is a system out of equilibrium where the proportion of customers
at any given table converges to zero almost surely. 
We show that for weight distributions in any of the three extreme
value classes, Weibull, Gumbel or Fr\'echet, the proportion of
customers sitting at the largest table converges  
to one in probability, {but not almost surely}, 
and the proportion of
customers sitting at either of the largest two tables converges to one
almost surely. 
\end{abstract}

\thanks{\copyright{} 2023 by the authors. This paper may be reproduced, in its entirety, for non-commercial purposes.}

\maketitle

\vspace{-.8cm}

\section{Introduction}

{Markets 
out of equilibrium  often follow a winner-takes-all dynamics by which competition allows the best performers to rise to the top at the expense of the losers~\cite{CF95}. In expanding markets,
as time passes, more competitive performers emerge and take the place of the current winner. In this paper we study  a simple model of this phenomenon, exploring the way in which new competitors take over from the current winners. In our model, a quality is attached to any product put on the market. When a new customer enters the market, a product is selected on the basis of its quality and on the number of customers that have chosen the product so far. This model is a disordered variant of the Chinese restaurant process of Dubins and Pitman~\cite{Pit06},
{\peter{see also~\cite{Aldous}}}. In this analogy 
customers enter a fictitious Chinese restaurant and choose a table to sit on; there is competition between tables in order to attract customers. We use this terminology throughout the paper.%
\smallskip

 More precisely, at first occupancy, a positive random 
`fitness', or `weight', is attached to each table, 
independently of everything else, according to
a fixed distribution~$\mu$. 
A new customer either joins an already occupied table, 
with probability proportional to both its fitness and the number of customers already sitting there, 
or sits at a new table, with probability proportional to a fixed parameter~$\theta>0$. The {proportion of customers at each table in the disordered Chinese restaurant process generates a dynamic random partition, representing the market share of each product in our earlier interpretation,
parametrised 
by a positive real number~$\theta$,
and a probability distribution~$\mu$ on the interval $(0,\infty)$. 
The aim of this paper is to understand the evolution of the 
largest tables in  the disordered Chinese restaurant process, representing the market share of the leading products.}
\pagebreak[3]\bigskip

In the classical model of~\cite{Aldous} and~\cite{Pit06},
the random partition (with elements in decreasing order)
converges in distribution to a Poisson-Dirichlet distribution of parameter~$\theta$.
For more information on the classical Chinese restaurant process, we refer the reader to, e.g., \cite{Pit06},  and the references therein. The introduction of the disorder radically changes the behaviour of the process
since, contrary to what happens in the classical case,
the proportion of customers sitting at any fixed table converges 
almost surely to zero as time goes to infinity.
This is because fitter and fitter tables keep entering the system.
This paper aims at answering the following questions:\smallskip

{\it What proportion of customers sit at the largest table at time $n$, i.e.\ when there are $n$ customers in the restaurant?
{\it What is the weight of this table?}
{\it When was this table first occupied?}}\smallskip

Our two main results are that:
\begin{itemize}
\item The proportion of customers sitting at the largest table converges to~one {\it in probability} as the number of customers grows to infinity, see Theorem~\ref{thm:A}. This result does not hold almost surely.
\item The proportion of customers sitting at the largest table {\it or at the second largest table} converges to~one {\it almost surely} as the number of customers grows to infinity, see Theorem~\ref{thm:B}.
\end{itemize}
We call Theorem~\ref{thm:B} the `two-table' theorem, 
as a reference to the parabolic Anderson `two-city' theorem, see~\cite{KLMS}.
Although the parabolic Anderson model is not at all related to the Chinese restaurant process, 
our results are reminiscent of those of~\cite{KLMS}, which they describe intuitively as follows:
``{\it at a typical large time, the mass, which is thought
of as a population, inhabits one site, interpreted as a city. At some rare times,
however, word spreads that a better site has been found, and the entire population
moves to the new site, so that at the transition times part of the population still lives
in the old city, while another part has already moved to the new one}''.
A similar interpretation holds in our setting, with tables replacing cities, and customers replacing the elements of the population.\smallskip

The proofs of our results rely on embedding 
the disordered Chinese restaurant process into continuous time. 
In this embedding, new tables are created at the jump times 
of a Poisson process of parameter~$\theta$, 
and the number of customers at each table is a Yule process 
whose parameter equals the weight of the table. This is reminiscent of the continuous-time embedding of the preferential attachment graph with fitnesses of Bianconi and Barab\'asi~\cite{BB01, Borgs}.
{ Our proof of Theorem~\ref{thm:A} relies on methods developed in~\cite{DMM, MMS} for the study of the  Bianconi-Barab\'asi model. It holds under a quite general assumption on the fitness distribution~$\mu$, we just ask that it belongs to an extreme value class, see Assumption~\ref{H}.
In particular, we allow the fitness distribution to have unbounded support. We are also able to give 
estimates of when the largest table at time~$n$ was first occupied, and of its weight.
 For the proof of Theorem~\ref{thm:B}, the `two-table theorem', a much refined analysis is needed.} Theorem~\ref{thm:B} holds under stronger assumptions on~$\mu$, see Assumptions~\ref{W},~\ref{F}, and~\ref{G}, depending on which extreme value class~$\mu$ belongs to; 
in Appendix~B we show that these assumptions are satisfied by a number of special cases of fitness distributions.
{We next give a formal definition of our model (Section~\ref{sub:def})
  and state our main results (Section~\ref{sub:main_results}).}

\subsection*{Acknowledgements}

The research of JEB was supported by 
\emph{Vetenskapsr{\aa}det, grant 2019-04185},
by \emph{Ruth och Nils Erik Stenb\"acks stiftelse}, 
and by the Sabbatical Program at the Faculty
of Science, University of Gothenburg.

\subsection{Mathematical definition of the model}\label{sub:def}
The weighted Chinese restaurant process
is a Markov process $(S_i(n)\colon i\geq 1)_{n\geq 0}$ 
taking values in the set of all sequences 
$(s_{i})_{i\geq 1}$ of nonnegative integers such that there exists $k\in\mathbb N$ with $s_i=0$ if and only if $i>k$. For all $n$, we call $S_i(n)$ the size of the $i$-th table at time $n$, 
and $K_n=\max \{i\geq 1\colon S_i(n)\neq 0\}$ the number of
occupied tables in the restaurant at time $n$. 
We sample a sequence $(W_i)_{i\geq 1}$ of i.i.d.\ random variables of distribution $\mu$, the weights or fitnesses. Given this sequence, the process is recursively defined. At time zero, $S_1(0) = 1$ and $S_i(n)=0$ for all $i\geq 2$.
Given the configuration at time~$n$, i.e.\ 
$(S_i(n))_{i\geq 1}$ either
\begin{itemize}[leftmargin=*]
\item the $(n+1)$-th customer enters the restaurant and
 sits at the $i$-th table, meaning that 
$S_i(n+1)=S_i(n)+1$ and $S_j(n+1)=S_j(n)$ for $j\neq i$,
 this happens with probability proportional to
 $W_iS_i(n)$; 
 \item or the $(n+1)$-th customer sits at a new
 table (table number $K_n+1$), meaning that 
$S_{K_n+1}(n+1)=1$ and $S_i(n+1)=S_i(n)$ for $i\leq K_n$,
with probability proportional to $\theta$.
\end{itemize}
The classical case of Pitman's process arises when the fitnesses are
deterministic, i.e.\ all tables have the same fitness. The case of
interest for us  is when $\mu$ has no mass at its essential supremum
(which may be finite or infinite) so  that fitter tables keep
emerging. Under this assumption, the following basic properties hold,
see Appendix~\ref{app:basic} for the proof. 
\bigskip

\begin{proposition}[Basic properties of the weighted Chinese Restaurant
  process]
\label{rem:basic}
\hfill
{\it \begin{itemize}[leftmargin=*]
\item[(i)] The number of occupied tables $K_n$ when the $n$th customer
  enters the restaurant satisfies 
\[
\lim_{n\to\infty} \frac{K_n}{\log n} = \frac\theta{{\rm essup } \, \mu}
\qquad \mbox{almost surely,}
\]
where the right hand side is interpreted as zero if the fitnesses are
unbounded.\smallskip 

\item[(ii)] For every $k\geq 1$
\[
S_k(n)\to\oo \qquad\mbox{and}\qquad
\frac{S_k(n)}{n}\to 0 \qquad\mbox{almost surely as }n\to\oo.
\]
Hence every fixed table has microscopic
 occupancy. 
\smallskip

\item[(iii)] There is no persistence of the table with maximal
  occupancy. In other words, the time $B_n$ at which the most occupied
  table at time $n$ gets its first occupany goes to infinity almost
  surely. 
  
  \smallskip

\item[(iv)] {The proportion of customers {\peter{ sitting at}} the largest table
\[
\max_{i\geq 1} \frac{S_i(n)}{n} 
\]
does not converge to one, almost surely.}  
\end{itemize}}
\end{proposition}

\subsection{Main results}\label{sub:main_results}

Here we briefly summarise our main results, postponing precise
formulations of our assumptions to the next section.  Our first result
is a `one-table-theorem' and states that, in probability, the largest
table `takes it all'.  It holds under Assumption~\ref{H}, stated
below, which essentially says that the weights $W_i$ belong to the
maximum domain of attraction of an extreme value distribution
(Weibull, Gumbel or Fr\'echet):

\begin{theorem}\label{thm:A}
Assume that the distribution $\mu$ of the weights $W_i$ satisfies
Assumption~\ref{H}, stated in Section~\ref{sec:H} below.  Then
\[
\max_{i\geq 1} \frac{S_i(n)}{n} \to 1,
\qquad \mbox{in probability as } n\to\oo.
\]
\end{theorem}

Recall that the convergence of Theorem~\ref{thm:A} does not hold almost surely. Our second main 
result states that there are never more than two tables of macroscopic
size.  For this result we need a strengthened version of our basic
Assumption~\ref{H}.

\begin{theorem}\label{thm:B}
Assume that the distribution $\mu$ of the weights $W_i$ satisfies
Assumption~\ref{W},~\ref{F}  or~\ref{G}, 
stated in Section~\ref{sec:H} below.  
Let $S^{(1)}(n)$ and  $S^{(2)}(n)$ denote the occupancy of the
largest two tables when
there are $n$ customers in the restaurant.
Then
\[
\frac{S^{(1)}(n)+S^{(2)}(n)}{n} \to 1,
\qquad \mbox{almost surely as } n\to\oo.
\]
\end{theorem}

Technically, it is more convenient to prove our main results for a
continuous-time version of our process and then transfer them to the
discrete-time process. We thus give the proofs of Theorems~\ref{thm:A} and~\ref{thm:B} at the end of Section~\ref{sec:CT}, in which we introduce the embedding of our process into continuous time and state their continuous-time analogues.

\section{The process in continuous time}  \label{sec:CT}

The disordered Chinese restaurant process is defined in the
introduction as a discrete time process. It can also be embedded into
continuous time and this embedding is a major technical tool for
us.\medskip

We first sample and fix a sequence $(W_i)_{i\geq 1}$ of 
i.i.d.\ random variables of distribution $\mu$, where $W_i$
constitutes the weight of table number~$i$. At time $t=0$, there is
one customer in the restaurant, sitting at table
number~1. Intuitively, given the weights $(W_i)_{i\geq 1}$, 
each customer sitting at table $i$ carries an exponential clock of
parameter~$W_i$, 
and when one of these clocks rings, a new customer enters the
restaurant and sits at the $i$-th table.  
In addition, customers enter the restaurant and open new tables at rate $\theta$. 
All exponential clocks are independent of each other.%
\medskip

More formally, we define $Z_i(t)$, the size of the $i$-th table at time $t$ in terms of an independent  {Yule process} $(Y_i(t))_{t\geq 0}$, where we recall that 
a \emph{Yule process} of parameter $\beta>0$ is a continuous-time branching process
where each individual is immortal and gives birth to one more individual at rate
$\beta$, independently of each other.  Writing
 $(Y_i)_{i\geq 1}$ for a sequence of i.i.d.\ Yule processes of
 parameter~one,  independent also of $(W_i)_{i\geq 1}$, we define
\be\label{eq:yule}
Z_i(t) = Y_i(W_i(t-\tau_i))\bs 1_{t\geq \tau_i},
\ee
where 
$\tau_0 = 0$ and the $\tau_i$'s for $i\geq 1$ are the jump-times of an independent 
Poisson counting process of rate~$\theta$. 
%
To see that $(Z_i(t)\colon i\geq 1)_{t\geq 0}$ is a continuous time  embedding of the discrete time process $(S_i(n)\colon i\geq 1)_{n\geq 0}$, we denote
$(\mathscr F_t\colon t\geq 0)$ the filtration generated by
$(Z_i(t)\colon i\geq 1)_{t\geq 0}$. Given $\mathscr F_t$ the next
change of the random vector $(Z_i(t)\colon i\geq 1)$ is either the
establishment of a new table if an exponential clock of
parameter~$\theta$ rings before the exponential clocks attached to the
customers already present ring, and this happens with probability
proportional to $\theta$, or the next customer joins an existing
customer at their table if her clock rings first, which happens with a
probability proportional to their table's fitness. 

The major advantage of the embedding comes from the fact that, by
elementary properties of the Yule process {\cec (see, e.g., \cite[Chapter~III]{AthreyaNey})}, 
{\cec there exists a sequence $(\zeta_i)_{i\geq 1}$ of i.i.d.\ random variables of exponential distribution of parameter~1 such that, for any fixed $i\geq 1$, $\mathrm e^{-t}Y_i(t) \to \zeta_i$ almost surely as $t\uparrow\infty$. 
Therefore,}
\be\label{eq:cv_yule}
Z_i(t) \sim \zeta_i \exp(W_i(t-\tau_i))\quad\text{ almost surely as }t\uparrow\infty.
\ee 
Thus the relative table sizes are primarily determined by the relative
sizes of the `exponents' $W_i(t-\tau_i)$.  This intuition is central
to much of our analysis and will be made rigorous later.

\subsection{Notation and setting}
\label{sec:H}

Recall that $\mu$ denotes the distribution of table weights.  We
assume that $\mu$ belongs to the maximum domain of attraction of a
distribution~$\nu$ on $\mathbb R$, meaning that 
there are functions $(A(t))_{t\geq 0}$ and   $(B(t))_{t\geq 0}$  such that
\be\label{eq:maxdomain}
\frac{\max_{i=1..n} W_i -A(n)}{B(n)} \Rightarrow \nu, 
\text{ in distribution as }n\to\infty.
\ee 
In fact, we assume the following.
If $\mu$ has bounded support, we assume without loss of generality
that its essential supremum is 1 and we define $M=1$.  If the
support of $\mu$ is unbounded, we set $M=\infty$. 
Throughout this paper we will assume that $\mu$ is
{absolutely continuous}.
Then our standing assumption is:

\begin{assumption}\label{H} 
{\bf (first part).} There are two continuous functions $(A(t))_{t\geq 0}$,
  $(B(t))_{t\geq 0}$ and a probability distribution $\nu$ 
on $\mathbb R$ such that, for all $x\in\mathbb R$, 
\be\label{eq:maxdomain-2}
t\mu\big((A(t)+xB(t),M) \big) \to -\log \nu(-\infty,x)=:\Phi(x), 
\quad\text{ as }t\uparrow\infty.
\ee
Also, $\Phi$ is non-increasing, $A$ is either constant or increasing, 
and either $A(t)= 0$ for all $t\geq0$, or
$A(t)/B(t)$ is non-decreasing and tends to infinity as
$t\uparrow\infty$. 
\end{assumption}

By classical extreme value theory, see for example 
\cite[Section 8.13]{BGT},  the stated properties of $A$ and $B$ hold without loss of
generality. Also 
$\nu$ is either a Weibull, a Gumbel, or a Fr\'echet distribution,
and we can choose $B$ non-negative  and 
$\Phi$ as in Table~\ref{cheatsheet}.
Also, we can choose 
$A = 1$ in the Weibull case, $A$ bounded from zero, increasing,
and converging to $M$ in the Gumbel case, 
and $A= 0$ in the Fr\'echet case.
Note that the Weibull case occurs only if $M=1$
and the Fr\'echet case only if  $M=\infty$, while  in the Gumbel
case we can have either $M=1$ or $M=\infty$.
\medskip

In the Weibull and Gumbel cases, to control the size of high-weighted tables that are created late in the process, we need the convergence of~\eqref{eq:maxdomain-2} to also hold in $L^1$.
This holds 
{\cec if the sequence of functions} $(u\mapsto n\mu\big((A(n)+uB(n),M))_{n\geq 1}$ is uniformly integrable, 
which is the case in all explicit examples we have considered, see also Appendix~B.
\addtocounter{theorem}{-1}
\begin{assumption} {\bf (continued)}.
If $\Phi$ is either the Weibull or the Gumbel distribution, then,
for all $x>0$, 
\be\label{eq:maxdomain-3}
\int_x^{\infty} t\mu\big((A(t)+uB(t),M) \big) \mathrm du \to \int_x^\infty \Phi(u)\mathrm d u, 
\quad\text{ as }t\uparrow\infty.
\ee
\end{assumption}
\smallskip

Further, in the  Weibull and Gumbel
cases, we define $u_t$ as the solution of
\be \label{eq:cond_u}
t B(u_t) = u_t A(u_t),
\ee
and we set $v_t = A(u_t), w_t = B(u_t)$. 
The existence of such $u_t$ is proved in Lemma~\ref{lem:ut} below. 
In particular, we have
\be
u_t v_t = tw_t.
\ee
In the Fr\'echet case, we set $u_t = t$, 
$v_t = 0$, and $w_t = B(t)$.
The motivation for these definitions is as follows:
\begin{itemize}[leftmargin=*]
\item The largest tables at time $t$ were 
created at times of order~$u_t$.
\item The weights of the largest tables
at time $t$ are of order $v_t + \Theta(w_t)$.
\end{itemize}
{\cec (See Theorem~\ref{thm:cv_gen} for a rigorous statement.)}

\begin{table}
\begin{center}
\begin{tabular}{|l|l|l|l|}
\hline\hline
&&& \\
{\bf Weibull} 
& $\Phi(x) = |x|^\alpha {\bf 1}_{x<0}$ 
& $u_t = t^{\frac\alpha{\alpha+1}} L_0(t)$
& $L_0(t)\to0$\\
&& $v_t=1$ &\\
&& $w_t = t^{-\frac1{\alpha+1}} L_0(t)$ &\\
&&&  \\
\hline
&&& \\
{\bf Gumbel} & $\Phi(x) = \e{-x}$ & $u_t = t L_1(t)$
& $L_1(t)\to0$ \\
&& $v_t = L_2(t)$ & $L_2(t)\to M$\\ 
&& $w_t = L_1(t)L_2(t)$ & \\
& && \\
\hline
&&& \\
{\bf Fr\'echet} & $\Phi(x) = \infty {\bf 1}_{x\leq 0} + x^{-\alpha}{\bf 1}_{x>0}$ 
& $u_t = t$ &\\
&& $v_t=0$ &\\
&& $w_t = t^{\frac1\alpha} L_3(t)$ & \\
&&& \\
\hline\hline
\end{tabular}\\
\end{center}
\caption{The functions $\Phi$, $u_t$, $v_t$ and $w_t$ for the three
  possible distributions $\nu$.
Here $\alpha>0$ and $L_0(t),\dots,L_3(t)$ denote slowly varying functions.} 
\label{cheatsheet}
\end{table}
\smallskip

Recall that a function $L(t)$ is 
called \emph{slowly varying} as $t\to\oo$
if $L(ct)/L(t)\to1$ as $t\to\oo$ for any fixed $c>0$.
A function $f(t)$ is called \emph{regularly varying of index $\beta$}
if $f(t)=t^\beta L(t)$ for some slowly varying $L(t)$.

\begin{lemma}\label{lem:ut}
Under Assumption~\ref{H}, in the Weibull and Gumbel cases, for all~$t$
large enough, 
Equation~\eqref{eq:cond_u} has a unique solution $u_t$. 
Furthermore, $u_t$ is non-decreasing in a neighbourhood of infinity,
$u_t\to\infty$, and $u_t = o(t)$ as $t\to\infty$. 
Further, 
\begin{itemize}[leftmargin=*]
\item[(i)] in the Weibull case, $v_t= 1$, and $(u_t)$ is regularly varying with index $\frac\alpha{\alpha+1}$ and  $(w_t)$ is regularly varying with index \smash{$\frac{-1}{\alpha+1}$};
\item[(ii)] in the Gumbel case $(u_t)$ is regularly varying of index 1, while 
$(v_t)$ and $(w_t)$ are slowly varying. Moreover,  $(v_t)$ is bounded
from zero for large enough $t$; 
\item[(iii)]  in the Fr\'echet case, $(w_t)$ is regularly varying with index $\frac1\alpha$.
\end{itemize}
\end{lemma}

\begin{proof}
By Assumption~\ref{H}, $A(u)/B(u)$ is non-decreasing in $u$, which
implies that the function $f(u)= uA(u)/B(u)$ is increasing to infinity.
This and continuity imply that, for all large enough~$t$, 
\eqref{eq:cond_u} has a unique solution $u_t=f^{-1}(t)$. 
Also, $u_t=f^{-1}(t)\uparrow \infty$ as $t\to\infty$.
Finally, $tB(u_t) = u_tA(u_t)$ implies that
${u_t}/t = {B(u_t)}/{A(u_t)} \to 0$,
as $t\uparrow\infty$.
\smallskip

$(i)$ By \cite[Theorem~8.13.3]{BGT} $B$ is regularly varying with index $-1/\alpha$. 
Hence, by~\cite[Theorem~1.5.12]{BGT}, the function 
$(u_t)$ is regularly varying with index \smash{$\frac\alpha{\alpha+1}$}. And as $w_t=B(u_t)$ we get that $(w_t)$ is regularly varying with index 
\smash{$\frac{-1}{\alpha+1}$}.
\smallskip

\pagebreak[3]
$(ii)$  Note that, in the Gumbel case, the functions $A(t)$
and $B(t)$ are both slowly varying.
This can be deduced from~\cite[Theorem~8.13.4]{BGT}
and its proof as follows. Using their notation, 
with $H(x)=-\log \mathbb P(X>x)$, 
we have that $A(t)=H^{\leftarrow}((\log 
t)+1)-H^{\leftarrow}(\log t)$.   This is slowly varying by condition 
(iii) in~\cite[Theorem~8.13.4]{BGT}.   
In the proof of the same theorem it is 
verified that $B(t)=H^{\leftarrow}(\log t)$ is in the de Haan class and 
therefore also slowly varying. \smallskip

$(iii)$ See \cite[Theorem~8.13.2]{BGT}.\end{proof}

We introduce the function 
\be\label{eq:Phi_t-def}
\Phi_t(x):=u_t \mu(v_t+x w_t,M)
\qquad \mbox{(where $\Phi_t(x)=0$ if  $v_t+x w_t\geq M$).}
\ee
By Assumption~\ref{H}, we have that
$\Phi_t(x)\to \Phi(x)$ as $t\to\oo$ for any $x\in\mathbb R$.
Also note that $\Phi_t(x)$ is decreasing in $x$ for any fixed $t$.
\medskip

Theorem~\ref{thm:B} requires {different} assumptions on $\mu$ than
Assumption~\ref{H}.
In the Weibull case, Assumption~\ref{H} implies that 
$\mu(1-x,1) =x^\alpha\ell(x)$ for some function $\ell$ that is slowly varying
at zero and some $\alpha>0$, see, e.g., \cite{Resnick}. 
We introduce the following stronger
assumption on $\alpha$ in this case.

\begin{assumption}
\label{W}{\bf (Weibull)}
$\mu$ is supported on $(0,1)$ and
$\mu(1-x,1) = x^\alpha\ell(x)$ where $\ell$  is slowly
varying {\peter at zero} and $\alpha>1$. 
\end{assumption}

Analogously to the Weibull case, in the Fr\'echet case Assumption 
\ref{H} implies that $\mu(x,\infty)$ is a regularly varying
function, this time at infinity. 
In this case, we actually 
do not need a stronger assumption on the index of
variation.  

\begin{assumption}\label{F}{\bf (Fr\'echet)}
$\mu$ is supported on $(0,\infty)$ and 
$\mu(x,\infty)=x^{-\alpha}L(x)$ 
where $L(x)$ is slowly varying at infinity and $\alpha> 0$.
\end{assumption}

In the Gumbel case, {the assumption needed for the two-table theorem} is more complicated.
Recall the function $\Phi_t(x)$ 
in \eqref{eq:Phi_t-def} 
and that $\Phi(x)=\e{-x}$ in this case.


\begin{assumption}\label{G}{\bf (Gumbel)}
{In addition to~\eqref{eq:maxdomain-2}}, 
\begin{enumerate}[leftmargin=*,label=(\roman*)]
\item\label{G2}
 There exist $c_1, c_2>0$ such that  for all $t$ large enough,
\[\begin{cases}
\Phi_t(x)\geq \mathrm e^{-x - c_1 x^2  /\log t} &
\text{ for all }x\in (-c_2 \log t, c_2 \log t),\\
\Phi_t(x)\leq \mathrm e^{- x + c_1 x^2 / \log t} &
\text{ for all }x\in (-c_2\log t, \tfrac{M-v_t}{w_t}).
\end{cases}
\]
\item\label{G3}
The slowly varying function $L_1(t)=w_t/v_t=u_t/t$ satisfies 
$$L_1(t) \log \log t\to0 \text{ as $t\to\oo$.}$$
\end{enumerate}
\end{assumption}

We give examples of distributions $\mu$ satisfying the 
assumptions in Appendix~\ref{app:examples}.
\pagebreak[3]

\subsection{Results in continuous time}

We let $M(t)$ denote the number of non-empty tables at time~$t$.
Recall that $\tau_n$ are the times of creation of tables, and
$W_n$ their weights.
The key step to get a one-table theorem in probability is to show
the following point process convergence.  
Recall the concept of vague convergence of measures:
if $\gamma,\gamma_1,\gamma_2,\dotsc$ are measures on a complete
separable metric space $\mathcal S$, then $\gamma_n$ converge vaguely
to $\gamma$ if $\int f\,\mathrm d\gamma_n\to\int f\, \mathrm d\gamma$ for all
non-negative, continuous, compactly supported functions 
$f \colon\mathcal S\to\mathbb R$.  The topology of vague convergence makes 
the set of Radon measures $\gamma$ on $\mathcal S$  a {Polish space.}
Thus the standard theory of 
convergence in distribution  applies
to random variables with values in this space. 
Let $\mathrm{PPP}(\lambda)$ denote a Poisson point 
process with $\sigma$-finite intensity measure $\lambda$, which is
represented as a random variable taking values in the set of Radon
measures. \pagebreak[3]

\begin{theorem}\label{thm:cv_gen}
Let
\[\mathcal S:= \begin{cases}
[0,\infty]\times [-\infty, \infty] \times (-\infty, \infty], & \text{ in the Weibull and Gumbel cases,}\\ 
[0,1]\times [0, \infty] \times (-\infty, \infty], &\text{ in the Fr\'echet case.}
\end{cases}\]
Under Assumption~\ref{H}, 
the random variables
\be\label{eq:Gamma}
\Gamma_t := \sum_{n=1}^{M(t)} \delta
\Big(\frac{\tau_n}{u_t}, \frac{W_n-v_t}{w_t}, \frac{\log Z_n(t)-tv_t}{tw_t}\Big)
\ee
taking values in the space of Radon measures on $\mathcal S$ equipped with the vague topology, converge in distribution as $t\to\infty$, to 
$\Gamma_\infty:=\mathrm{PPP}(\mathrm d\zeta(s, y, z))$, where
\[\mathrm d\zeta(s, y, z):=
\begin{cases}
\theta\mathrm ds\otimes 
-\Phi'(y)\mathrm dy \otimes \delta_{y-s}(\mathrm dz)
&\text{ in the Weibull and Gumbel cases,}\\
\theta\mathrm ds\otimes 
-\Phi'(y)\mathrm dy \otimes \delta_{y(1-s)}(\mathrm dz)
&\text{ in the Fr\'echet case.}
\end{cases}\]
\end{theorem}
The proof of this theorem appears at the end of Section~\ref{sec:PPPcv}.
It shows that the largest tables at time~$t$
were created around time $\Theta(u_t)$, have fitness of order
$v_t+\Theta(w_t)$, and thus, their size at time~$t$ is of order
$\exp(tv_t+\Theta(tw_t))$.  
Indeed, the mass of all points with $\tau_n/u_t\to 0$ {(corresponding to ``older'' tables)}
concentrates asymptotically on the subset of $\mathcal S$ where the first coordinate is zero. 
As this set has no mass under the intensity measure of the limiting Poisson process, 
these points must leave every compact subset of~$\mathcal S$ 
and, {\cec because of the compactification of the intervals in  the
definition of~$\mathcal S$}, 
this can only happen by their third coordinate going to $-\infty$. 
Hence none of these points corresponds to the largest table. 
This argument, {\cec which is crucial in the proof,} also applies when 
$\tau_n/u_t\to \infty$, or $\frac{W_n-v_t}{w_t}$ goes to infinity, or to zero in the 
Fr\'{e}chet, or $-\infty$ in the Weibull or Gumbel case. 
\medskip

As promised,  the point process convergence of Theorem~\ref{thm:cv_gen} 
implies a one-table theorem in probability.
\begin{corollary}[One-table theorem]
\label{cor:one_table}
Let $N(t)$ denote the number of customers in the restaurant at time~$t$.
If Assumption~\ref{H} holds, then
\be\label{eq:onetab}
\max_{1\leq i \leq M(t)}\frac{Z_i(t)}{N(t)} \to 1, \text{ in probability when
}t\to\infty.
\ee
\end{corollary}

\begin{proof}
Let $Z^{(1)}(t)$ and $Z^{(2)}(t)$ 
denote the sizes of the largest and second largest tables at time $t$.
Also let
\[
W^{(1)}(t) = \frac{\log Z^{(1)}(t)-tv_t}{tw_t}\quad \text{ and }\quad
W^{(2)}(t) = \frac{\log Z^{(2)}(t)-tv_t}{tw_t}.
\]
By Theorem~\ref{thm:cv_gen}, we have, for all $z_1, z_2>0$,
\ba\mathbb P\big(W^{(1)}(t)\geq z_1, W^{(2)}(t)\geq z_2\big)
&= \mathbb P\big(\Gamma_t(\widehat{\mathcal S}\times [z_1, \infty])\geq 1, \Gamma_t(\widehat{\mathcal S}\times [z_2, \infty])\geq 2\big)\\
&\to \mathbb P\big(\Gamma_\infty(\widehat{\mathcal S}\times [z_1, \infty])\geq 1, \Gamma_\infty(\widehat{\mathcal S}\times [z_2, \infty])\geq 2\big)
\ea
where we have set
\[\widehat{\mathcal S}
= \begin{cases}
[0,\infty]\times[-\infty, \infty]& \text{ in the Weibull and Gumbel cases,}\\
[0,1]\times[0,\infty]& \text{ in the Fr\'echet case.}
\end{cases}\]
This implies that, as $t\to\infty$, we have 
$(W^{(1)}(t), W^{(2)}(t)) \Rightarrow (W^{(1)}, W^{(2)}),$
where $W^{(1)}$ and $W^{(2)}$ are two almost-surely finite random variables
such that $W^{(1)}>W^{(2)}$ almost surely.
Clearly
\be\label{eq:N(t)_negl}
N(t) = \sum_{i=1}^{M(t)} Z_i(t) = Z^{(1)}(t) + \bigg(\sum_{i=1}^{M(t)}
Z_i(t)-Z^{(1)}(t)\bigg). 
\ee
Our aim is to show that the second term is negligible in front of $Z^{(1)}(t)$.
Almost surely for all $t\geq 0$,
\[
0\leq 
\sum_{i=1}^{M(t)} Z_i(t)-Z^{(1)}(t)\leq M(t) Z^{(2)}(t) 
= M(t) Z^{(1)}(t)\exp\big[\big(W^{(2)}(t)-W^{(1)}(t)\big)tw_t\big].
\]
Since $M(t)$ is Poisson-distributed with parameter $\theta t$, {\cec we have}
$W^{(2)}(t)-W^{(1)}(t)\Rightarrow W^{(2)}-W^{(1)}<0$, and $\log t = o(tw_t)$ , 
we indeed get that, in probability as $t\uparrow\infty$,
\[\sum_{i=1}^{M(t)} Z_i(t)-Z^{(1)}(t) = o\big(Z^{(1)}(t)\big),\]
which, by~\eqref{eq:N(t)_negl}, implies $N(t) = (1+o(1)) Z^{(1)}(t)$ and thus concludes the proof.
\end{proof}

From Corollary~\ref{cor:one_table} it is a small step to Theorem
\ref{thm:A}:
\begin{proof}[Proof of Theorem~\ref{thm:A}]
Writing $T_n$ for the time of arrival of the $n$-th customer, we have 
$S_i(n)=Z_i(T_n)$. 
Then $T_n\to\infty$ almost surely as $n\to\infty$ 
(indeed, $\{\sup_{n\geq 1}T_n<\infty\}$ is equivalent to $\{\exists t_\infty\colon N(t_\infty) = \infty\}$, which has probability zero), 
so that $\max_{i\geq 1} S_i(n)/n=\max_{1\leq i\leq M(T_n)} Z_i(T_n)/N(T_n)\to1$ in probability.
\end{proof}


The following
result states that, almost surely as $t\uparrow\infty$, no more than
two tables can have macroscopic sizes at time~$t$. 

\begin{theorem}\label{th:two_tables}
Assume that  Assumption~\ref{W}~{(Weibull)},
 Assumption~\ref{F}~{(Fr\'echet)} or  Assumption~\ref{G}~{(Gumbel)} hold.
 Denote by $Z^{(1)}(t)$ the size of the largest table at 
 time~$t$, and by $Z^{(2)}(t)$ the size of the second largest table at
time $t$.
 Then,
\[
\frac{Z^{(1)}(t)+Z^{(2)}(t)}{N(t)}\to 1,
\qquad \mbox{ almost surely as $t\to\infty$,}
\]
where $N(t)$ is the total number of 
customers in the restaurant at time~$t$.
\end{theorem}

The proof of this theorem is given in Section~\ref{sec:twotabproof}.
We now show how to deduce Theorem~\ref{thm:B}.

\begin{proof}[Proof of Theorem~\ref{thm:B}]
As in the proof of Theorem~\ref{thm:A}, we let $T_n$ be the time of arrival of the $n$-th customer; we have $T_n\uparrow\infty$ almost surely as $n\uparrow\infty$, and $S_i(n) = Z_i(T_n)$, for all $n\geq 1$, $i\geq 1$.
Thus, by Theorem~\ref{th:two_tables},
\[\frac{S^{(1)}(n)+S^{(2)}(n)}{n}
=\frac{Z^{(1)}(T_n)+Z^{(2)}(T_n)}{N(T_n)}\to 1,
\]
almost surely as $n\uparrow\infty$.
\end{proof}

It remains to prove Theorems~\ref{thm:cv_gen} and~\ref{th:two_tables}; this is done in Sections~\ref{sec:PPPcv} and~\ref{sec:twotabproof} below, respectively.

\section{One-table result: Proof of 
Theorem~\ref{thm:cv_gen}
}\label{sec:PPPcv}

The proof  of Theorem~\ref{thm:cv_gen} 
is done in two steps.
Firstly, in Subsection~\ref{sub:local} 
we prove convergence of $\Gamma_t$ (see~\eqref{eq:Gamma})
on the space of measures on 
\begin{equation}\label{eq:defW}
\mathcal W :=
\begin{cases}
[0,\infty)\times (-\infty, \infty] \times [-\infty, \infty], & \text{in the Weibull and Gumbel cases,}\\
[0,1)\times (0, \infty] \times [-\infty, \infty], & \text{in the Fr\'echet case.}
\end{cases}
\end{equation}
Note that this differs from the claim of Theorem~\ref{thm:cv_gen},
where convergence is on the space of measures on the space 
$\mathcal S$ which differs from $\mathcal W$ at the endpoints of
several of the intervals.
Secondly, {\cec and this is the most difficult part of the proof}, 
in Subsection~\ref{sub:young_and_unfit}, 
we prove that young tables 
($\tau_i\gg u_t$) as well as unfit tables 
($W_i-v_t\ll w_t$)
are both too small to contribute to the limit.
This allows us to `close the brackets' in the first two coordinates
of \eqref{eq:defW};
in doing so however, the mass corresponding to tables that do not
contribute to the limit instead `escapes' to $-\infty$ in the third
coordinate.   We thereby transfer
the convergence on $\mathcal W$ to convergence on $\mathcal S$.

\subsection{Local convergence}\label{sub:local}

We prove the following convergence for the space $\mathcal W$.

\begin{lemma}\label{lem:local}
In distribution as  $t\to\infty$,
\[\Gamma_t 
\to \mathrm{PPP}(\mathrm d\zeta(s,y,z)),\]
where 
\[\mathrm d\zeta(s,y,z) = \begin{cases}
\theta\mathrm ds\otimes 
-\Phi'(y)\mathrm dy \otimes \delta_{y-s}(\mathrm dz),
&\text{ in the Weibull and Gumbel cases,}\\
\theta\mathrm ds\otimes 
-\Phi'(y)\mathrm dy \otimes \delta_{y(1-s)}(\mathrm dz),
&\text{ in the Fr\'echet case.}
\end{cases}
\]
on the space of measures on $\mathcal W$ equipped with the vague
topology.
\end{lemma}

To prove Lemma~\ref{lem:local}, we first prove that
\begin{equation}\label{eq:cv_Psit}
\Psi_t := \sum_{n=1}^{M(t)} \delta
\Big(\frac{\tau_n}{u_t}, \frac{W_n-v_t}{w_t}, \frac{W_n(t-\tau_n)- tv_t}{tw_t}\Big)
\to \mathrm{PPP}\big(\mathrm d\zeta(s, y, z)\big),
\end{equation}
on $\mathcal W$, 
and then prove that this implies convergence of $\Gamma_t$ on the same
space.  The difference between $\Gamma_t$ and $\Psi_t$ is that in the
third coordinate we have replaced $\log Z_n(t)$ with its conditional
mean  $W_n(t-\tau_n)$.  
We will show that Equation~\eqref{eq:cv_Psit} is a direct consequence of the following lemma:
\begin{lemma}\label{lem:hat_Psi}
For all $t\geq 0$, in distribution, as $t\to\infty$,
\[\widehat \Psi_t:=\sum_{n=1}^{M(t)} \delta
\Big(\frac{\tau_n}{u_t}, \frac{W_n-v_t}{w_t}\Big)
\to \mathrm{PPP}\big(\theta\mathrm ds\otimes 
-\Phi'(y)\mathrm dy\big),
\]
on the space of measures on $\widehat{\mathcal W}$ equipped with the vague topology,
where
\[\widehat{\mathcal W}
:=\begin{cases}
[0,\infty)\times (-\infty, \infty], & \text{in the Weibull and Gumbel cases,}\\
[0,1)\times (0, \infty],  & \text{in the Fr\'echet case.}
\end{cases}
\]
\end{lemma}

{Before proving Lemma~\ref{lem:hat_Psi}, we show how to deduce~\eqref{eq:cv_Psit} from it: If we set $s_{n,t} = \tau_n/u_t$ and $y_{n,t} = (W_n-v_t)/w_t$, then
\ban
W_n(t-\tau_n) 
&= (v_t + y_{n,t} w_t)(t-s_{n,t}u_t)\notag \\
&= tv_t + y_{n,t} t w_t - s_{n,t} u_tv_t - s_{n,t} y_{n,t} u_tw_t.\label{eq:trick}
\ean
In the Weibull and Gumbel cases, we have $u_tv_t = tw_t$, and thus
\ba
W_n(t-\tau_n) &= tv_t + (y_{n,t}-s_{n,t}) tw_t - s_{n,t} y_{n,t} u_tw_t,
\ea
which implies
\[\frac{W_n(t-\tau_n) -tv_t}{tw_t} = y_{n,t}-s_{n,t} - s_{n,t} y_{n,t} \frac{u_t}t.\]
Because $u_t/t \to 0$ as $t\uparrow\infty$, 
this concludes the proof of~\eqref{eq:cv_Psit} in the Weibull and Gumbel cases.
In the Fr\'echet case, because $u_t = t$ and $v_t = 0$, \eqref{eq:trick} gives
\[W_n(t-\tau_n) 
= y_{n,t} t w_t - s_{n,t} y_{n,t} t w_t = y_{n,t}(1-s_{n,t}) t w_t,\]
which concludes the proof of~\eqref{eq:cv_Psit}.}

\begin{proof}[Proof of Lemma~\ref{lem:hat_Psi}]
Invoking Kallenberg's theorem~\cite[Prop.\ 3.22]{Resnick}, it is enough to prove that for all compact boxes
$B=[0, a]\times [b, \infty]$,
where $b\in\mathbb R$ in the Weibull and
Gumbel cases, and $b>0$ in the Fr\'echet case, we have
\begin{itemize}
\item $\mathbb P(\widehat \Psi_t(B)=0)
\to\exp\big(\int_B \theta \Phi'(y) \mathrm ds\mathrm dy\big) 
= \exp(-\theta a \Phi(b)),$
\item $\mathbb E[\widehat \Psi_t(B)]
\to -\int_B \theta \Phi'(y)\mathrm ds\mathrm dy
= \theta a \Phi(b)$.
\end{itemize}
We let $I(t)$ be the set of all $n$ such that $\tau_n\leq au_t$; 
so that $|I(t)|$ is Poisson-distributed with parameter $a\theta u_t$.
We have
\[\mathbb P(\widehat \Psi_t(B)=0)
= \mathbb P\big(\forall 1\leq n\leq |I(t)|, W_n < v_t+bw_t\big)
=\mathbb E\Big[\big(1-\mu(v_t+bw_t, M)\big)^{|I(t)|}\Big],
\]
where we recall that $M\in\{1,\infty\}$ is the essential supremum of~$\mu$.
Since $|I(t)|$ is Poisson-distributed with parameter $a\theta u_t$, we get
\[\mathbb P(\widehat \Psi_t(B)=0)
= \exp\big(-a\theta u_t \mu(v_t+bw_t, M)\big)
= \exp\big(-a\theta u_t \mu\big(A(u_t)+bB(u_t), M\big)\big),\]
since, by definition, $v_t = A(u_t)$ and $w_t = B(u_t)$.
By Assumption~\ref{H}, 
\[\mathbb P(\widehat \Psi_t(B)=0)
\to \exp\big(-a\theta \Phi(b)\big),\]
as $t\uparrow\infty$,
which concludes the proof of the first assumption of Kallenberg's theorem.
For the second assumption, note that
\[\begin{split}
\mathbb E[\widehat \Psi_t(B)]
&= \mathbb E\Big[\sum_{n\in I(t)} \bs 1_{W_n\geq v_t+bw_t}\Big]
= \mathbb E\big[|I(t)|\big]\mu(v_t+bw_t, M)\\
&= a\theta u_t \mu\big(A(u_t)+bB(u_t), M\big)
\to a\theta \Phi(b),
\qquad \mbox{as } t\uparrow\infty,
\end{split}\]
by Assumption~\ref{H}. This concludes the proof.
\end{proof}

Lemma~\ref{lem:local} is an immediate consequence of the following
result and Lemma~\ref{lem:hat_Psi}, which established convergence of
$\Psi_t$. 
\begin{lemma}\label{lem:Psi_vs_Gamma}
For all continuous,
compactly supported functions $f : \mathcal W\to \mathbb R$, we have
\[\bigg|\int f\, \mathrm d\Gamma_t - \int f\, \mathrm d\Psi_t\bigg| \to 0,
\]
in distribution when $t\to\infty$.
\end{lemma}

\begin{proof}
First note that, by density of the set of Lipschitz-continuous,
compactly supported functions in the set of continuous, 
compactly supported functions with respect to $L^\oo$-norm, 
we may assume that $f$ is Lipschitz-continuous.
Let $a>0$ and $b\in\mathbb R$ in the Weibull and Gumbel 
cases, respectively $a\in[0,1)$ and  $b>0$ in the Fr\'echet case, 
and let $f: [0,a]\times [b, \infty] \times [-\infty,\infty]$ 
be a Lipschitz-continuous function of Lipschitz constant $\kappa$.
We have
\[\bigg|\int f\, \mathrm d\Gamma_t - \int f\, \mathrm d\Psi_t\bigg|
\leq \kappa\sum_{n\in I(t)} \frac{|\log Z_n(t)-W_n(t-\tau_n)|}{tw_t},
\] 
where $I(t)$ is the set of all integers $n$ such that
\[\tau_n \in [0, a u_t]
\quad\text{ and }\quad
W_n\geq v_t + b w_t.\]
For all $n\geq 1$ and $s\geq 0$, 
we set $R_n(s)= \sup_{t\geq s} |\log Y_n(t) - t|$, 
where we recall from \eqref{eq:yule} that 
$Y_n$ is the Yule process such that
$Z_n(t)=Y_n(W_n(t-\tau_n))\bs 1_{t\geq \tau_n}$.
 By definition, 
$v_t +bw_t\to M\geq 1$, and $u_t \leq t$ (see
Equation~\eqref{eq:cond_u}).
This means that there is $\delta>0$ such that
 $W_n(t-\tau_n)\geq (v_t+bw_t)(t-au_t)\geq \delta t$ 
for all $t$ large enough (we can take $\delta=(1-a)/2$ in the
Fr\'echet case and $\delta=\nicefrac12$ in the Weibull and Gumbel
cases). 
We thus get that, almost surely for all $t$ large enough,
\[
\bigg|\int f\, \mathrm d\Gamma_t - \int f\, \mathrm d\Psi_t\bigg|
\leq \frac{\kappa}{t w_t} \sum_{n\in I(t)} R_n(\delta t).
\]
For all integers~$n$, note that $R_n(\delta t)\to |\log \zeta_n|$
almost surely as $t\uparrow\infty$.
Moreover, $\liminf_{t\to\infty} t w_t = \infty$.
Since, in addition, by Lemma~\ref{lem:hat_Psi} and its proof, 
$|I(t)|=\hat \Psi_t([0,a]\times[b,\infty])$ 
converges in distribution to an almost-surely finite random variable independent of $(\zeta_n)$
this concludes the proof.
\end{proof}

\subsection{New and unfit tables do not contribute}
\label{sub:young_and_unfit}

To get convergence of $\Gamma_t$ on $\mathcal S$ rather than
$\mathcal W$
we prove that ``new'' tables, as well as tables with small weight, 
are too small to contribute to the limit. 
We start with the new tables.
\begin{lemma}\label{lem:young}
For all $\varepsilon, \kappa>0$, there exists {\peter $x_0<M$} such that, for all sufficiently large $t$, {\cec for all $x\geq x_0$,}
\[\mathbb P\Big(\max_{n\leq M(t)} \big(\log Z_n(t)\big) \bs 1_{\{\tau_n\geq x u_t\}}\geq 
\ell_\kappa(t)\Big)\leq \varepsilon,\]
where 
\begin{equation}\label{eq:def_ell}
\ell_\kappa(t)
:=\begin{cases}
tv_t-\kappa tw_t & \text{ in the Weibull and Gumbel cases,}\\
\kappa tw_t &\text{ in the Fr\'echet case.} 
\end{cases}
\end{equation}
\end{lemma}

\begin{proof}
Recall the Yule processes $Y_n$ from \eqref{eq:yule}.
For all $n\geq 1$, set 
\be
A_n = \sup_{s\geq 0} Y_n(s)\mathrm e^{-s} 
= \sup_{s\geq \tau_n} Z_n(s)\mathrm e^{-W_n(s-\tau_n)}.
\ee
Note that the $A_n$ are i.i.d.\ and that
$A_n$ is in fact independent of $W_n$ 
 as it only depends on $Y_n$.  Let 
$A=\sup_{s\geq 0} Y(s)\mathrm e^{-s}$ be a random variable with the
distribution of the~$A_n$.  Then we have the following tail-bound:
for some $C>0$
\be\label{eq:star}
\mathbb P(A>u)\leq C\e{-u/2},
\qquad \mbox{for all } u>0.
\ee
This {\peter{is}} proved using the maximal inequality for the submartingale 
$\exp(\theta Y(s)\e{-s})$, where $\theta\in(0,1)$, and that
$\mathbb E[\exp(\theta Y(s)\e{-s})]$ is uniformly bounded,
which may be verified using the explicit distribution, 
$\mathbb P(Y(s)=k)=\e{-s}(1-\e{-s})^{k-1}$ for~$k\geq1$.%
\smallskip

Let $I_x(t)$ be the set of all integers~$n$ such that $\tau_n\geq x u_t$; 
using a union bound in the second inequality, we get
\ba
\mathbb P\Big(\max_{n\leq M(t)}& \log Z_n(t)\bs 1\{\tau_n\geq x u_t\} 
\geq \ell_\kappa(t)
\Big)\\
&\leq \mathbb P\big(\exists n\in I_x(t)\colon A_n\geq \exp\big(\ell_\kappa(t) - W_n (t-\tau_n)\big)\\
&\leq \mathbb E\bigg[\sum_{n\in I_x(t)} \mathbb P\Big(A_n\geq \exp\big(\ell_\kappa(t) - W_n (t-\tau_n)\big) \Big| (\tau_n) \Big)\bigg].
\ea 
As $(\tau_n)_{n\geq 1}$ is a Poisson process of parameter $\theta$, 
independent of $(A_n)$ and $(W_n)$,
\[
\mathbb P\Big(\max_{n\leq M(t)} \log Z_n(t)\bs 1\{\tau_n\geq x u_t\} 
\geq \ell_\kappa(t)
\Big)
\leq   \theta\,\int_{xu_t}^t\mathrm ds \,  \mathbb P\Big(A\geq \exp\big(\ell_\kappa(t) - W(t-s)\Big) ,
\]
where $A$ is a copy of $A_1$ and $W$ a copy of $W_1$, independent of each other.
Thus, 
\be\begin{split}\label{eq:before_cases}
\mathbb P\Big( &\max_{n\leq M(t)} \log Z_n(t)\bs 1\{\tau_n\geq x u_t\} \geq \ell_\kappa(t)
\Big)\\
&\leq  \theta\, \int_{xu_t}^t \mathrm ds\,  \int_0^{\infty}  \mathrm d\mu(w) \, \mathbb P\big(A\geq \exp\big(\ell_\kappa(t) - w(t-s)\big)\big)\\
&=  \theta\, \int_{x}^{t/u_t} \mathrm da\,   \int_{-v_t/w_t}^{\infty} \mathrm d\tilde\mu_t(u)
\mathbb P\big(A\geq \exp\big(\ell_\kappa(t) -
(v_t+uw_t)(t-au_t)\big),
\end{split}\ee
where $\mathrm d\tilde\mu_t(u) := u_t  \mathrm d\mu(v_t+uw_t)$ and 
we have used the changes of variable $s = au_t$ and $w=v_t+uw_t$.
We treat the rest of the proof separately for the Weibull and Gumbel cases on the one hand, 
and the Fr\'echet case on the other hand.\smallskip

{\bf The Weibull and Gumbel cases:}
In these cases, $\ell_\kappa(t) = tv_t -\kappa tw_t$ and 
$u_tv_t = tw_t$, which implies that
\ban
&\mathbb P\Big(\max_{n\leq M(t)} \log Z_n(t)\bs 1\{\tau_n\geq x u_t\} 
\geq \ell_\kappa(t)
\Big)\notag\\
&\leq  \theta \int_{x}^{t/u_t} \mathrm da  \int_{-v_t/w_t}^{\infty}\mathrm d\tilde\mu_t(u)\, 
\mathbb P\big(A\geq \exp\big(-(\kappa+u)tw_t + au_t (v_t+uw_t)\big)\big)\notag\\
&\leq \theta \int_{x}^{t/u_t} \mathrm da  \int_{-v_t/w_t}^{\infty} \mathrm d\tilde\mu_t(u)\, 
\bs 1\{a(v_t+uw_t)\leq (2\kappa+u)v_t\}\notag\\
&\hspace{1cm}+ \theta \int_{x}^{t/u_t} \mathrm da  \int_{-v_t/w_t}^{\infty} \mathrm d\tilde\mu_t(u)
 \bs 1\{a(v_t+uw_t)>(2\kappa+u)v_t\}
\mathbb P\big(A\geq \mathrm e^{\kappa tw_t}\big) \notag\\
&\leq \theta \int_{x}^{t/u_t} \!\!\!\mathrm da \int_{-v_t/w_t}^{\infty} \!\!\mathrm d\tilde\mu_t(u)\, 
\bs 1\{a(v_t+uw_t)\leq (2\kappa+u)v_t\} + C\mathrm e^{-\frac12\exp(\kappa tw_t)}tw_t.
\label{eq:bla}
\ean
{In the last step, we used that there exists a constant $C>0$ such that
$\mathbb P(A\geq u)\leq C\mathrm e^{-\nicefrac u2}$ for all $u\geq 0$,
and  also that $\int_{-\infty}^{\infty} \mathrm d\tilde
\mu_t(u) = u_tw_t$. }
Since $tw_t\to\oo$, we get that the second term
above tends to zero as $t\uparrow\infty$. 
{For the first term, note that, for all $a<t/u_t = v_t/w_t$,
\[a(v_t+uw_t)\leq (2\kappa+u)v_t
\quad\Leftrightarrow\quad
u\geq \frac{a-2\kappa}{1-aw_t/v_t}
\quad\Rightarrow\quad 
u\geq a-2\kappa,\]
and thus, for all $x>2\kappa$,
\ban
&\theta \int_{x}^{t/u_t}\mathrm da  \int_{-v_t/w_t}^{\infty}\mathrm d\tilde\mu_t(u)\,  \bs 1\{a(v_t+uw_t)
\leq (2\kappa+u)v_t\}\notag\\
&\leq \theta \int_{x}^{\infty} \mathrm da \int_{a-2\kappa}^{\infty} \mathrm d\tilde\mu_t(u)
= \theta \int_{x}^{\infty} \mathrm da\, \Phi_t(a-2\kappa)
\to \theta \int_{x}^{\infty} \mathrm da \, \Phi(a-2\kappa),
\label{eq:young_up_bound}
\ean
as $t\uparrow\infty$, by Assumption~\ref{H}, see~\eqref{eq:maxdomain-3}. 
\smallskip

We look at the two different possibilities for $\Phi$: 
in the Weibull case, 
$\Phi$ is zero on $(0, \infty)$,
and thus
$\int_{x}^{\infty} \theta\mathrm da \Phi(a-2\kappa)=0$
as soon as $x>2\kappa$.
In the Gumbel case, 
we have $\Phi(u) = \mathrm e^{-\alpha u}$ for some $\alpha>0$, 
and thus
\[ \int_{x}^{\infty} \theta\mathrm da \Phi(a-2\kappa)
=\int_{x}^{\infty} \theta\mathrm da \mathrm e^{-\alpha(a-2\kappa)}
= \frac1{\alpha} \mathrm e^{-\alpha(x-2\kappa)},\]
which tends to zero as $x\to\infty$.
In both the Weibull and Gumbel cases, 
we thus get that for all $\delta>0$, for all~$x$ large enough,
$\int_{x}^{\infty} \theta\mathrm da \Phi(a-2\kappa)\leq \nicefrac\delta2.$
Therefore, by~\eqref{eq:bla} and~\eqref{eq:young_up_bound}, for all $x$ large enough, for all $t$ large enough,
\begin{equation}\label{eq:concl}
\mathbb P\Big(\max_{n\leq M(t)} \log Z_n(t)\bs 1\{\tau_n\geq x u_t\} \Big)
\leq \delta,
\end{equation}
which concludes the proof.}
\medskip

{\bf The Fr\'echet case:}
In the Fr\'echet case, $v_t =0$, $u_t= t$, and $\ell_\kappa(t) = \kappa tw_t$.
Thus, \eqref{eq:before_cases} becomes
\be\begin{split}
\mathbb P\Big(\max_{n\leq M(t)} &\log Z_n(t)\bs 1\{\tau_n\geq x t\} \Big)\\
&\leq \theta \int_x^1 \mathrm da \int_0^{\infty} \mathrm d\tilde\mu_t(u)
\mathbb P\big(A\geq \exp\big((\kappa-(1-a)u) t w_t\big) \\
&\leq \theta  \int_x^1 \mathrm da \int_0^{\infty}  \mathrm d\tilde\mu_t(u) \bs 1\{(1-a)u\geq \kappa/2\}\\
&\qquad+ C\mathrm e^{-\frac12\exp(\kappa tw_t/2)} 
\theta \int_x^1 \mathrm da  \int_0^{\infty} \mathrm d\tilde\mu_t(u) \bs 1\{(1-a)u<\kappa/2\}
\notag\\
&\leq \theta \int_0^{\infty} \big(1-\tfrac{\kappa}{2u}-x\big)_{\!+} \mathrm d\tilde\mu_t(u)
+ C\theta tw_t \mathrm e^{-\frac12\exp(\kappa tw_t/2)},\label{eq:F2terms_Frechet}
\end{split}\ee
because $\tilde\mu(0,\infty) = tw_t$.
The second term goes to zero as $t\uparrow\infty$ for all $\kappa>0$.
For the first term, we get
\ba
\int_0^{\infty} \theta\Big(1-\tfrac{\kappa}{2u}-x\Big)_{\!+} \mathrm d\tilde\mu_t(u)
&\leq \theta\int_{\frac{\kappa}{2(1-x)}}^{\infty} \mathrm d\tilde\mu_t(u)
= \theta\Phi_t\big(\tfrac{\kappa}{2(1-x)}\big)\\
&=(\theta+o(1)) \Phi\big(\tfrac{\kappa}{2(1-x)}\big),
\ea
as $t\uparrow\infty$, by Assumption~\ref{H}.
Thus, making $x$ close to~1, one can make the first term of~\eqref{eq:F2terms_Frechet} as small as desired, which concludes the proof in the Fr\'echet case.
\end{proof}

{\cec In the following lemma, we control the contributions of the small-weight tables:}
\begin{lemma}\label{lem:unfit}
For all $\varepsilon, \kappa>0$, there exists {\cec $y_0$} such that, for all sufficiently large $t$, {\cec for all $y\geq y_0$},
\[\mathbb P\Big(\max_{n\leq M(t)} \log Z_n(t)\bs 1\{W_n\leq v_t - yw_t\} 
\geq \ell_\kappa(t)\Big)\leq \varepsilon,\]
where $\ell_\kappa(t)$ is defined in~\eqref{eq:def_ell}.
\end{lemma}
\pagebreak[3]

\begin{proof}
This proof is very similar to the proof of the previous lemma. Note that,  for all $n\geq 1$, $t\geq 0$, 
if $\log Z_n(t)\geq\ell_\kappa(t)$ and $W_n\leq v_t-y w_t$, 
then
\ba
\log Z_n(t)-W_n(t-\tau_n)
&\geq \ell_\kappa(t) - (v_t-yw_t)t\\
&=\begin{cases}
(y-\kappa)tw_t & \text{ in the Weibull and Gumbel cases},\\
(y+\kappa)tw_t & \text{ in the Fr\'echet case.}
\end{cases}
\ea
Therefore, using the independence of $M(t)$ and $(A_n)$,
\[\mathbb P\Big(\max_{n\leq M(t)} \log Z_n(t)\bs 1\{W_n\leq v_t-y w_t\} \geq \ell_\kappa(t)\Big)
\leq \mathbb E[M(t)]\mathbb P\big(A_1\geq \exp((y-\kappa)tw_t)\big).\]
Recall that $M(t)$ is Poisson distributed of parameter $\theta t$, and thus
\[\mathbb P\Big(\max_{n\leq M(t)} \log Z_n(t)\bs 1\{W_n\leq v_t-y w_t\} \geq \ell_\kappa(t)\Big)
\leq C_0\theta t \exp\big(-\tfrac12 \exp((y-\kappa)tw_t)\big),
\]
where we used that $\PP(A\geq x)\leq C_2 \mathrm e^{-x/2}$.
Since $w_t\to\infty$,
in the Fr\'echet case, $t$ can be made large enough so that $C_0\theta
t \exp(-\frac12 \exp((y+\kappa)tw_t))\leq \varepsilon$. In the Weibull
and Gumbel cases, 
for all $y>\kappa$, $t$ can be made large enough so that $C_0\theta t \exp(-\frac12 \exp((y-\kappa)tw_t))\leq \varepsilon$. This completes the proof in all three cases. 
\end{proof}

We now show how to deduce Theorem~\ref{thm:cv_gen} from
Lemmas~\ref{lem:local}, \ref{lem:young} and~\ref{lem:unfit}: 
\begin{proof}[Proof of Theorem~\ref{thm:cv_gen}]
We give details of the proof in the Weibull and Gumbel cases, as the Fr\'echet case is identical, except that the first coordinate takes values in $[0,1]$ instead of $[0, \infty]$, and the third in $(0,\infty]$ instead of $(-\infty,\infty]$.
\medskip

Let $f \colon [0,\infty]\times [-\infty,\infty]\times (-\infty,\infty]\to \mathbb R$ be a non-negative, continuous and compactly supported function. 
Let $\kappa>0$ such that $\{f\neq 0\}\subseteq [0,\infty]\times [-\infty,\infty]\times [-\kappa,\infty]=: \mathcal A(\kappa)$.
We aim to prove that, in distribution as $t\uparrow\infty$,
\be
\int f \mathrm d\Gamma_t \to \int f \mathrm d\Gamma_{\infty}\ee
Fix $\eta>0$. 
By Lemma~\ref{lem:young}, there exists {\cec $x_0 = x_0(\kappa, \eta)$ such that, for all $x\geq x_0$,}
\be\label{eq:B}
\liminf_{t\uparrow\infty}\mathbb P\big(\Gamma_t(\mathcal B(x,\kappa)) = 0\big)\geq 1-\eta,\ee
where we have set $\mathcal B(x,\kappa) = (x,\infty]\times[-\infty,\infty]\times [-\kappa, \infty]$.
Furthermore, by Lemma~\ref{lem:unfit}, 
there exists {\cec $y_0 = y_0(\kappa, \eta)$ such that, for all $y\geq y_0$,}
\be\label{eq:C}
\liminf_{t\uparrow\infty}\mathbb P\big(\Gamma_t(\mathcal C(y,\kappa)) = 0\big)\geq 1-\eta,\ee
where we have set $\mathcal C(y,\kappa) = [0,\infty]\times[-\infty, -y)\times [-\kappa, \infty]$.
For all $t\geq 0$,
\[\int f \mathrm d\Gamma_t
= \int_{\mathcal A(\kappa)} f \mathrm d\Gamma_t
= \int_{\mathcal A(\kappa)\cap\mathcal B(x,\kappa)^c\cap \mathcal C(y,\kappa)^c} f \mathrm d\Gamma_t + R(t),\]
where 
\[0\leq R(t)
\leq \int_{\mathcal B(x,\kappa)} f \mathrm d\Gamma_t +\int_{\mathcal C(y,\kappa)} f \mathrm d\Gamma_t.\]
By~\eqref{eq:B} and~\eqref{eq:C}, for all $t$ large enough, with probability at least $1-2\eta$, 
$\Gamma_t(\mathcal B(x,\kappa)) = \Gamma_t(\mathcal C(y,\kappa)) = 0$, implying that $R(t) = 0$.
To conclude, note that
\[\mathcal A(\kappa)\cap\mathcal B(x,\kappa)^c\cap \mathcal C(y,\kappa)^c
= (x,\infty]\times [-\infty,-y)\times[-\kappa,\infty],
\]
and thus, by Lemma~\ref{lem:local}, in distribution as $t\to\infty$,
\[\int_{\mathcal A(\kappa)\cap\mathcal B(x,\kappa)^c\cap \mathcal C(y,\kappa)^c} f \mathrm d\Gamma_t\to \int_{\mathcal A(\kappa)\cap\mathcal B(x,\kappa)^c\cap \mathcal C(y,\kappa)^c} f \mathrm d\Gamma_\infty.\]
Making $x$ and $y$ large enough, because $\Gamma_{\infty}$ has no atom, we can make the right-hand side arbitrarily close to $\int_{\mathcal A(\kappa)}f \mathrm d\Gamma_\infty = \int f \mathrm d\Gamma_\infty$, which concludes the proof.
\end{proof}

\section{Two-table theorem: 
Proof of Theorem~\ref{th:two_tables}}
\label{sec:twotabproof}

For the proof
of Theorem~\ref{th:two_tables} we treat the three cases (Weibull,
Gumbel and Fr\'echet) in parallel.  
Although technical details differ, the general strategy is the same
for all cases.  
We first work on the `exponents' instead of the table sizes.  That is,
we set, for all $t\geq 0$ and all $1\leq i\leq M(t)$,
\begin{equation}\label{eq:def_theta}
\Theta_n(t) :=W_n(t-\tau_n).
\end{equation}
Recall from~\eqref{eq:cv_yule} that
$Z_n(t) \sim \zeta_n \exp(\Theta_n(t))$ almost surely as $t\uparrow\infty$, 
where $(\zeta_n)_{n\geq 1}$ is a sequence of i.i.d.\ random variables of 
exponential distribution of parameter~1.
This is why we call the $\Theta_n(t)$ the `exponents'.
We also introduce the order statistics of this sequence,
$\Theta^{(1)}(t)\geq \Theta^{(2)}(t)\geq \Theta^{(3)}(t)\geq \ldots$
and we let $m_i=m_i(t)$ be the index such that
$\Theta^{(i)}(t)=\Theta_{m_i(t)}(t)$.  
Then $\tau_{m_i(t)}$ denotes the time of creation of the
table which at time $t$ has the $i$-th largest exponent.
In what follows we often
suppress the
$t$-dependence of $m_i(t)$ from the notation.\medskip

Recall the function $(w_t)_{t\geq 0}$ given in Lemma~\ref{lem:ut}.
In this section, we establish the following result, which, 
by the Borel--Cantelli lemma,
gives the existence of a diverging sequence of times $(t_k)_{k\geq 0}$
at which, almost surely, 
the largest and third-largest exponents are well separated.

\begin{figure}
\begin{center}
\includegraphics[width = 7cm]{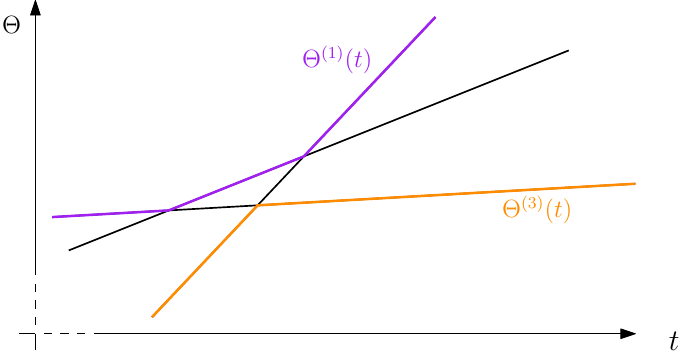}
\end{center}
\caption{Schematic picture of the largest exponents at a time of transition. 
In Proposition~\ref{prop:two_tables_summable}, we bound the gap between the largest exponent $\Theta^{(1)}(t)$ (in purple) and third largest exponent $\Theta^{(3)}(t)$ (in  orange).}
\label{fig:exponents}
\end{figure}

\begin{proposition}\label{prop:two_tables_summable}
Let $t_k=k^\eta$ and $\lambda_t=t^{-\kappa}$, where $\eta,\kappa>0$
satisfy $2\kappa \eta>1$.  
Then under Assumption~\ref{H} for the Weibull and Fr\'echet cases,
respectively Assumption~\ref{G} for the Gumbel case, we have
\be\label{eq:theta-sep-sum}
\sum_{k=1}^\oo \PP\big(\Theta^{(1)}(t_k)-\Theta^{(3)}(t_k)\leq 
\lambda_{t_k} t_k w_{t_k}\big)<\oo.
\ee
\end{proposition}


We prove Proposition~\ref{prop:two_tables_summable} in
Section~\ref{sub:bounds_exponents}.
Here is a brief summary of how, in Sections~\ref{sub:interpolation}
to~\ref{sub:proof_main}, we conclude the 
proof of Theorem~\ref{th:two_tables} once Proposition~\ref{prop:two_tables_summable} has been established.  
{We argue in two steps that the fraction $\sum_{j=3}^{M(t)}{Z_{m_j}(t)}/{Z_{m_1}(t)}$ converges to zero
almost surely. First, in Section~\ref{sub:interpolation} we show that it suffices to consider the process at times $(t_k)_{k\geq 1}$, which are sufficiently dense and, by \eqref{eq:theta-sep-sum}, at these times, the largest and third-largest exponents are well-separated.
Second, in Section~\ref{sub:approxi}, we show that $Z_{m_j}(t)$ indeed
grows like $\exp(\Theta^{(j)}(t))$, 
using a large deviations estimate for Yule processes given in Lemma~\ref{lem:Peter}. Therefore the fraction $\sum_{j=3}^{M(t)}{Z_{m_j}(t_k)}/{Z_{m_1}(t_k)}$
is bounded by
$M(t) \exp(\Theta^{(3)}(t_k)-\Theta^{(1)}(t_k))$ and by
\eqref{eq:theta-sep-sum} 
the exponent is smaller than
$-\lambda_{t_k}t_kw_{t_k}$ almost surely for all $k$ large enough.  Finally, in Section~\ref{sub:proof_main},
we deduce the same result for 
\smash{$\sum_{j=3}^{M(t)}{Z_{n_3}(t)}/{Z_{n_1}(t)}$},
where $n_i = n_i(t)$ is the index of the $i$-th largest table at
time~$t$ (which may be different from the index of the~$i$-th largest
exponent).}


\subsection{Potter bounds}
In the proofs, the  following {Potter bounds} for slowly varying
functions  will be useful, see Theorem~1.5.6(i) in \cite{BGT}. 
\begin{itemize}[leftmargin=*]
\item If $L(x)$ is positive and slowly varying as $x\to\oo$, 
then for any $\delta, C_1, C_2>0$, there exists $x_0 = x_0(\delta, C_1, C_2)>0$ such that, 
for all $x\geq x_0$,
\be\label{eq:potter-infty}
C_1 x^{-\delta}\leq L(x) \leq C_2 x^\delta.
\ee
\item If $\ell(x)$ is positive and slowly varying as $x\to0$, 
then, for any $\delta, c_1, c_2>0$, there exists $x_0 = x_0(\delta, c_1, c_2)>0$ such that, 
for all $|x|\leq x_0$,
\be\label{eq:potter-zero}
c_1 x^{\delta}\leq \ell(x) \leq c_2 x^{-\delta}.
\ee
\end{itemize}
Below, we will typically write $L, L_1, L_2, \dotsc$ for functions slowly
varying at infinity and $\ell, \ell_1, \ell_2, \dotsc$ for functions slowly
varying at zero.

\subsection{Proof of Proposition~\ref{prop:two_tables_summable}}\label{sub:bounds_exponents}

To prove \eqref{eq:theta-sep-sum} we consider the following normalised
version of the exponents.   For all 
$t\geq 0$, $1\leq n\leq M(t)$, let
\begin{equation}\label{eq:def_xi}
\xi_n(t) = \frac{W_n(t-\tau_n)-tv_t}{tw_t}
=\frac{\Theta_n(t)-t v_t}{t w_t}
\end{equation}
and introduce also their order statistics,
$\xi^{(1)}(t)\geq \xi^{(2)}(t)\geq \xi^{(3)}(t)\geq \ldots$
We aim  to find an upper
bound for $\mathbb P(\xi^{(1)}(t)-\xi^{(3)}(t)\leq \lambda_t)$.
Note that the $\xi_n(t)$ are all negative in the
Weibull case (since $v_t=1$), all positive in the Fr\'echet case
(since $v_t=0$), and can be either positive or negative in the
Gumbel case.   \smallskip

For all $t\geq 0$ and $x\in\mathbb R$, we let
\begin{equation}\label{eq:defA}
A_t(x) =\big\{(s,w)\in[0,t]\times[0,M): w(t-s) > tv_t + x t w_t\big\}.
\end{equation}
Then the event that $\xi_n(t) > x$ is the same as the event that
$(\tau_n,W_n)\in A_t(x)$.   We let 
$\Pi := ((\tau_n,W_n))_{n\geq 1}$, which
is a Poisson point process on 
$[0,\oo)\times[0,M)$.   We write
$\pi:= \theta\mathrm ds \otimes \mathrm d\mu$ 
for the intensity measure of $\Pi$.
\smallskip

Recall from \eqref{eq:Phi_t-def} that $\Phi_t(x)=u_t \mu(v_t+x w_t,M)$, which is non-increasing in~$x$.

\begin{lemma}\label{lem:piA}
Let $x> -v_t/w_t$.  Then
\be\label{eq:pi-A}
\pi(A_t(x))=
\theta(v_t+x w_t)\frac{t w_t}{u_t}
\int_x^\oo
\frac{\Phi_t(z)}{(v_t+z w_t)^2}\, \dd z, 
\ee
and for $\eps>0$, 
\be\label{eq:pi-A-diff}
0\leq \pi(A_t(x)) -  \pi(A_t(x+\eps))
\leq
\eps \theta \frac{tw_t}{u_t}
\frac{\Phi_t(x) }{v_t+x w_t}.
\ee
\end{lemma}

\begin{proof}
For \eqref{eq:pi-A} we use that
\ban
\pi(A_t(x))
&= \int_0^t \int_0^\infty \theta \, \mathrm ds\mathrm d\mu(w) \bs 1_{(s,w)\in A_t(x)}
= \int_0^t \theta\,\mathrm ds \mu\Big(\frac{tv_t+xtw_t}{t-s}, M\Big)\notag\\
&= \theta (v_t+xw_t)\int_{x}^\infty  
\frac{tw_t\mu(v_t+zw_t,  M)}{(v_t+zw_t)^2}\,\mathrm dz \\
&=\theta(v_t+x w_t)\tfrac{t w_t}{u_t}
\int_x^\oo
\frac{\Phi_t(z)}{(v_t+z w_t)^2}\, \dd z,
\ean
where we used the change of variable
\smash{$\tfrac{tv_t+xtw_t}{t-s} = v_t+zw_t$,} 
to go from $s$ to $z$.
For \eqref{eq:pi-A-diff}, using 
$A_t(x+\eps)\subseteq A_t(x)$ and \eqref{eq:pi-A}, discarding a term which
is $\leq0$, we get
\be
0\leq \pi(A_t(x)) -  \pi(A_t(x+\eps))
\leq\theta(v_t+x w_t)\tfrac{t w_t}{u_t}
\int_x^{x+\eps}
\frac{\Phi_t(z)}{(v_t+z w_t)^2}\, \dd z.
\ee
We then use the fact that 
the integrand is non-increasing in $z$ (because $z\geq x>-v_t/w_t$) to get the result.
\end{proof}

\begin{lemma}\label{lem:largest-xi}
Under Assumption~\ref{H} we have the following bounds.
\begin{itemize}[leftmargin=*] 
\item In the Weibull case, 
let $x_t>0$ such that $x_t w_t\to0$,
{\cec for any $\delta, C>0$, for all $t$ large enough,}
\be\label{eq:pi-A-W}
\pi(A_t(-x_t))
\geq C x_t^{1+\alpha+\delta} 
t^{-\delta(1+\alpha+\delta+\frac1{1+\alpha})},
\ee
and, whenever $\xi-\lambda\geq -x_t$, we have
\be\label{eq:pi-A-diff-W}
\pi(A_t(\xi-\lambda))-\pi(A_t(\xi))
\leq C\lambda x_t^{\alpha-\delta} 
t^{\delta(1+\alpha+\delta+\frac1{1+\alpha})}.
\ee
\item In the Fr\'echet case, let $x_t>0$ such that $x_tw_t\to\oo$,
{\cec for any $\delta, C>0$, for all $t$ large enough,}
\be\label{eq:pi-A-F}
\pi(A_t(x_t))
\geq C x_t^{-(\alpha+\delta)} 
t^{-\delta(\alpha+\frac1\alpha+\delta)},
\ee
and, whenever $\xi-\lambda\geq x_t$, we have
\be\label{eq:pi-A-diff-F}
\pi(A_t(\xi-\lambda))-\pi(A_t(\xi))
\leq C \lambda x_t^{-(1+\alpha-\delta)} 
t^{\delta(1/\alpha+\alpha-\delta)}.
\ee
\item In the Gumbel case, 
let $x_t>0$ such that $x_t w_t/v_t\to 0$,
if $M=1$ then $\frac{x_t w_t}{1-v_t}\to0$,
\be\label{eq:pi-A-G}
\pi(A_t(-x_t))
\geq \big(\theta+o(1)\big) \int_{-x_t}^{x_t} \Phi_t(z) \dd z,
\ee
as $t\uparrow\infty$,
and, whenever $\xi-\lambda\geq -x_t$, we have
\be\label{eq:pi-A-diff-G}
\pi(A_t(\xi-\lambda))-\pi(A_t(\xi))
\leq C \lambda \Phi_t(-x_t).
\ee
\end{itemize}
\end{lemma}

\begin{proof}
We argue  separately for the three
cases.  It is helpful to refer to Table~\ref{cheatsheet}.
\begin{itemize}[leftmargin=*]
\item In the Weibull case, $M=1$ and $v_t=1$.  Then 
$\Phi_t(z)=0$ as soon as $z\geq 0$.  Also $u_t=tw_t$.
For any $-1/w_t<x<0$, by~\eqref{eq:pi-A},
\be
\pi(A_t(x))=\theta(1+x w_t)\int_x^0 \frac{\Phi_t(z)}{(1+zw_t)^2} \dd z
\geq \theta(1+x w_t)\int_x^0 \Phi_t(z) \dd z.
\ee
Replacing $x$ with $-x_t$ 
and using  that $1-x_tw_t=1+o(1)$
we get
\be
\pi(A_t(-x_t))\geq (\theta+o(1))  \int_0^{x_t} \Phi_t(-z) \dd z.
\ee
Now recall that $\mu(1-\eps,1)=\eps^{\alpha} \ell(\eps)$,
$u_t=t^{\frac{\alpha}{1+\alpha}} L(t)$ and 
$w_t=t^{-\frac{1}{1+\alpha}} L(t)$ (see Assumption~\ref{W} and Lemma~\ref{lem:ut}). 
This, combined with
the Potter bounds~\eqref{eq:potter-infty} and \eqref{eq:potter-zero}, gives
for $0\leq z\leq x_t$ 
\be\begin{split}
\Phi_t(-z)
&=z^\alpha L(t)^{1+\alpha}  \ell(z w_t)
\geq C_1 z^\alpha t^{-\delta(1+\alpha)}
(zw_t)^{\delta} \\
&\geq C_2 z^{\alpha+\delta}
t^{-\delta(1+\alpha+\delta+\frac1{1+\alpha})}.
\end{split}\ee
Then 
\be
\pi(A_t(-x_t))\geq C_2(\theta+o(1)) 
t^{-\delta(1+\alpha+\delta-\frac1{1+\alpha})}
\int_0^{x_t} z^{\alpha+\delta} \dd z,
\ee
so \eqref{eq:pi-A-W} follows.  For \eqref{eq:pi-A-diff-W},
we have from  \eqref{eq:pi-A-diff} that
\be\begin{split}
\pi(A_t(\xi-\lambda))-\pi(A_t(\xi))
&\leq \theta \lambda \frac{\Phi_t(\xi-\lambda)}{1+(\xi-\lambda)w_t}
\leq \theta \lambda \frac{\Phi_t(-x_t)}{1-x_tw_t}\\
&=(\lambda\theta +o(1)) \Phi_t(-x_t).
\end{split}\ee
Similarly to~\eqref{eq:similar}, the Potter bounds~\eqref{eq:potter-infty} give
\be
\Phi_t(-x_t)\leq C_3 x_t^{\alpha-\delta} 
t^{\delta(1+\alpha-\delta+\frac1{1+\alpha})},
\ee
which gives \eqref{eq:pi-A-diff-W}.
\item In the Fr\'echet case,
$M=\infty$,
$u_t=t$, $v_t=0$ and  $w_t=t^{\frac1\alpha}L(t)$, so
for $x>0$, 
\eqref{eq:pi-A} simplifies to
\be\label{eq:Fr_piA}
\pi(A_t(x))
= \theta  x \int_x^{\infty} \frac{\Phi_t(z)}{z^2} \dd x.
\ee
Moreover, $\mu(x,\infty)=x^{-\alpha}L_1(x)$.
Using the Potter bounds
\eqref{eq:potter-infty} 
we get, for any $\delta>0$, $z\geq x_t$, and $t$ large enough,
\be\begin{split}
\Phi_t(z)&=t(zw_t)^{-\alpha} L_1(zw_t)
\geq C_1 t (zw_t)^{-\alpha-\delta}
= C_1 z^{-\alpha-\delta} t^{-\delta/\alpha}L(t)^{-\alpha-\delta}\\
&\geq C_2 z^{-\alpha-\delta} 
t^{-\delta/\alpha-\delta(\alpha+\delta)}.
\end{split}\label{eq:similar}
\ee
Then 
\be\begin{split}
\pi(A_t(x_t))
&\geq   C_2(\theta+o(1))  x_t
t^{-\delta/\alpha-\delta(\alpha+\delta)}
\int_{x_t}^{\infty} z^{-(2+\alpha+\delta)} \dd u \\
&\geq C_3 t^{-\delta/\alpha-\delta(\alpha+\delta)}
x_t^{-(\alpha+\delta)},
\end{split}\ee
as claimed in \eqref{eq:pi-A-F}.  Next, from \eqref{eq:pi-A-diff}
we have 
\be
\pi(A_t(\xi-\lambda))-\pi(A_t(\xi)) \leq
\lambda \theta \frac{\Phi_t(\xi-\lambda)}{\xi-\lambda}
\leq 
\lambda \theta \frac{\Phi_t(x_t)}{x_t}.
\ee
Similarly to~\eqref{eq:similar}, using the Potter bounds
\eqref{eq:potter-infty} 
\be
\Phi_t(x_t)\leq C_4 x_t^{-\alpha+\delta} 
t^{\delta(1/\alpha+\alpha-\delta)},
\ee
which gives \eqref{eq:pi-A-diff-F}.\pagebreak[3]

\item In the Gumbel case, we use that $u_tv_t=tw_t$ and that
$\Phi_t(z)=0$ if $z\geq (M-v_t)/w_t$ to see that, 
by~\eqref{eq:pi-A}, for any $0<x<\tfrac{v_t}{w_t}$,
\be \label{eq:G_piAtx}
\pi(A_t(-x))
=\theta  (1-x\tfrac{w_t}{v_t})
\int_{-x}^{(M-v_t)/w_t} \frac{\Phi_t(z)}{(1+z\tfrac{w_t}{v_t})^2} 
\mathrm dz. 
\ee
We now note that $x_t\leq (M-v_t)/w_t$ for all $t$ large enough.
This is immediate if $M=\oo$, while 
if $M=1$ then it follows from the assumption 
$x_tw_t/(1-v_t)\to0$.
Since the integrand in \eqref{eq:G_piAtx}
is non-negative, we  get that, as $t\uparrow\infty$,
\be
\pi(A_t(-x_t))
\geq \frac{\theta(1-x_tw_t/v_t)}{(1+x_t{w_t}/{v_t})^2}
\int_{-x_t}^{x_t} \Phi_t(z) \mathrm dz
\geq \big(\theta+o(1)\big) \int_{-x_t}^{x_t} \Phi_t(z) \mathrm dz,
\ee
because $(1+z\tfrac{w_t}{v_t})^{-2}\geq (1+x_t\tfrac{w_t}{v_t})^{-2}$ for all $z\leq x_t$, and because
$x_t = o(v_t/w_t)$ as $t\uparrow\infty$. 
Next, from \eqref{eq:pi-A-diff},
we get that, for all $\xi$ and $\lambda$ such that $\xi-\lambda\geq -x_t$,
\be\begin{split}
\pi(A_t(\xi-\lambda))-\pi(A_t(\xi)) &\leq
\lambda \theta \frac{\Phi_t(\xi-\lambda)}{1+(\xi-\lambda)\frac{w_t}{v_t}}
\leq 
\lambda \theta
\frac{\Phi_t(-x_t)}{1-x_t\frac{w_t}{v_t}}\\
&\leq
\lambda (\theta+o(1)) \Phi_t(-x_t),
\end{split}\ee
as $t\uparrow\infty$,
as required for \eqref{eq:pi-A-diff-G}.
\qedhere
\end{itemize}
\end{proof}

Using Lemma~\ref{lem:largest-xi}, we deduce the following
key estimates on $\xi^{(1)}(t)-\xi^{(3)}(t)$.

\begin{lemma}\label{lem:two-tables-bound}
Under Assumption~\ref{H}  for the Weibull and Fr\'echet cases,
and Assumption~\ref{G} for the Gumbel case,
let $\lambda_t=t^{-\kappa}$ where $\kappa>0$.  Then
for any $\gamma, C>0$, for all $t$ large enough,
\be\label{eq:xi31-bd}
\mathbb P\big(\xi^{(1)}(t)-\xi^{(3)}(t)\leq \lambda_t\big)
\leq C t^\gamma\lambda_t^2.
\ee
\end{lemma}


\begin{proof}
For any $y_t\in\mathbb R$ we have the decomposition
\ban
&\mathbb P\big(\xi^{(1)}(t)-\xi^{(3)}(t)\leq \lambda_t\big)\notag\\
&\hspace{2cm}\leq \mathbb P(\xi^{(1)}(t)\leq y_t) + 
\mathbb P\big(\xi^{(1)}(t)-\xi^{(3)}(t)\leq \lambda_t 
\text{ and } \xi^{(1)}(t)> y_t\big).
\label{eq:two_terms}
\ean
We will use this for $y_t>-v_t/w_t+\lambda_t$.
Note that
\be
\mathbb P(\xi^{(1)}(t)\leq y_t) 
=\mathbb P(\Pi(A_t(y_t)) = \varnothing) 
=\exp(-\pi(A_t(y_t))).
\ee
For the other term in \eqref{eq:two_terms}, 
note that  $(\xi_n(t))_{n\geq 0}$ is a Poisson point process
and let
$\rho_t(\cdot)$ denote its intensity measure.  Then
using Mecke's formula, see e.g.\ \cite[Th.\ 4.1]{LP},
and simple properties of Poisson random variables,
we have for all~$t>0$,  
\be\begin{split}\label{eq:mecke}
\mathbb{P}(&\xi^{(1)}(t)-\xi^{(3)}(t)\leq\lambda_t
\text{ and } \xi^{(1)}(t)> y_t)\\
&=\int_{y_t}^{\oo} \mathrm d\rho_t(\xi) 
\mathbb{P}(\Pi(A_t(\xi))=\varnothing)
\mathbb{P}(|\Pi(A_t(\xi-\lambda_t))\setminus
\Pi(A_t(\xi))| \geq 2)\\
&\leq \int_{y_t}^{\oo} \mathrm d\rho_t(\xi) 
\mathbb{P}(\Pi(A_t(\xi))=\varnothing)
(\pi(A_t(\xi-\lambda_t))-
\pi(A_t(\xi)))^2.
\end{split}\ee
The intuition behind the first equality is that we integrate over all possible values $\xi$ for $\xi^{(1)}(t)$. For $\xi$ to be maximal, there needs to be one point of the point process at $\xi$, and none larger (hence the term $\mathbb{P}(\Pi(A_t(\xi))=\varnothing)$); for 
$\xi^{(1)}(t)-\xi^{(3)}(t)\leq\lambda_t$, there needs to be at least two points in $A_t(\xi-\lambda_t))\setminus A_t(\xi)$. (See Figure~\ref{fig:pic}.)
\begin{figure}
\begin{center}\includegraphics[width = 6cm]{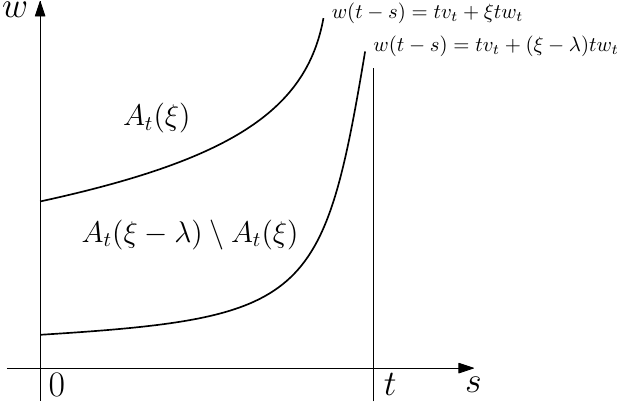}\end{center}
\caption{Intuition behind~\eqref{eq:mecke}.}
\label{fig:pic}
\end{figure}
Note that, using Mecke's formula again,
\be\label{eq:atmost1}
\int_{y_t}^{\oo} \mathrm d\rho_t(\xi) 
\mathbb{P}(\Pi(A_t(\xi))=\varnothing)=
\mathbb P(\xi^{(1)}(t)> y_t)\leq 1.
\ee\pagebreak[3]

We proceed using Lemma~\ref{lem:largest-xi}.
\begin{itemize}[leftmargin=*]
\item In the Weibull case, we set $x_t=t^\eps$
for \smash{$0<\eps<\tfrac1{1+\alpha}$} and in the decomposition
\eqref{eq:two_terms} we set $y_t=-\tfrac12x_t$.  
As \smash{$w_t=t^{-\frac1{1+\alpha}}L(t)$} we have $x_tw_t\to0$.
Then, by~\eqref{eq:pi-A-W}, for all $t$ large enough
\be\label{eq:xi1-W}
\mathbb P(\xi^{(1)}(t)\leq y_t)\leq 
\exp(-C t^{(\eps-\delta)(1+\alpha+\delta)
-\frac\delta{1+\alpha}}).
\ee
Using \eqref{eq:pi-A-diff-W} in 
\eqref{eq:mecke} and applying \eqref{eq:atmost1} we get that, for all $t$ large enough,
\be\label{eq:mecke-W}
\mathbb{P}(\xi^{(1)}(t)-\xi^{(3)}(t)\leq\lambda_t
\text{ and } \xi^{(1)}(t)> y_t)\leq
C \lambda_t^2 
t^{2\eps(\alpha-\delta)+2\delta(1+\alpha+\delta
+\frac1{1+\alpha})}.
\ee
Clearly we may select $\eps,\delta>0$ small enough that
\eqref{eq:xi1-W} and \eqref{eq:mecke-W} are both at most 
$Ct^\gamma \lambda_t^2$, for any $\gamma>0$.\smallskip
\item  In the Fr\'echet case, we set $x_t=t^{-\eps}$
for $0<\eps<\tfrac1{\alpha}$  and $\eps<\kappa$ and we set 
$y_t=x_t$.  Since \smash{$w_t=t^{\frac1\alpha}L(t)$}, we have $x_tw_t\to\oo$.
By \eqref{eq:pi-A-F}, for all $t$ large enough,
\be\label{eq:xi1-F}
\mathbb P(\xi^{(1)}(t)\leq y_t)\leq 
\exp(-C t^{\eps(\alpha+\delta)-\delta(\alpha+1/\alpha+\delta)}
)
\ee
Using \eqref{eq:pi-A-diff-F} in 
\eqref{eq:mecke} and applying \eqref{eq:atmost1} we get that, for all $t$ large enough,
\be\label{eq:mecke-F}
\mathbb{P}(\xi^{(1)}(t)-\xi^{(3)}(t)\leq\lambda_t
\text{ and } \xi^{(1)}(t)> y_t)\leq
C \lambda_t^2 
t^{2\eps(1+\alpha-\delta)
+2\delta(1/\alpha+\alpha-\delta)}.
\ee
Clearly we may select $\eps,\delta>0$ small enough that
\eqref{eq:xi1-F} and \eqref{eq:mecke-F} are both at most 
$Ct^\gamma \lambda_t^2$, for any $\gamma>0$.\smallskip
\item In the Gumbel case, 
let us for ease of reference recall Assumption~\ref{G}\ref{G2}:
\be\label{eq:G2-Phi_t}
\begin{cases}
\Phi_t(x)\geq \e{-x - c_1 x^2 /\log t}, &\text{ for all }
x\in (-c_2 \log t, c_2 \log t)\\
\Phi_t(x)\leq \e{- x + c_1 x^2 /\log t}, &\text{ for all }
x\in (-c_2 \log t, \tfrac{M-v_t}{w_t}).
\end{cases}
\ee
Set $x_t=2 \log\log t$.
By Assumption~\ref{G}\ref{G3}, $x_tw_t/v_t\to 0$ and thus
Lemma~\ref{lem:largest-xi} applies; together with 
the lower bound in \eqref{eq:G2-Phi_t}, this gives, for all~$t$ large enough,
\be\begin{split}
\mathbb{P}(\xi^{(1)}(t)\leq-x_t)&\leq
\exp\Big[-(\theta+o(1))\int_{-x_t}^{x_t} \Phi_t(z) \mathrm dz\Big]\\
&\leq \exp\Big[-(\theta+o(1))\exp\big(- c_1 x_t^2 /\log t\big) 
\int_{-x_t}^{x_t} \mathrm e^{- z}\mathrm dz\Big].
\end{split}\ee
Since $x_t^2/\log t\to 0$  we get that, as $t\uparrow\infty$,
\be\label{eq:G_first_term}
\mathbb{P}(\xi^{(1)}(t)\leq -x_t)
\leq\exp\big(-(\theta+o(1)) \mathrm e^{x_t}\big)
=\mathrm e^{-(\theta+o(1))(\log t)^2}
\leq \lambda_t^2,
\ee
for all $t$ large enough, because $\lambda_t = t^{-\kappa}$.
Next, \eqref{eq:pi-A-diff-G} gives that, for $\xi-\lambda_t\geq -x_t$,
\ban
\pi(A_t(\xi-\lambda_t))-\pi(A_t(\xi))
&\leq \lambda_t (\theta+o(1))\Phi_t(-x_t)
\leq \lambda_t (\theta+o(1)) \e{x_t+c_1 x_t^2 /\log t}\notag\\
&= \lambda_t \exp\big(x_t(1+o(1))\big),
\ean
because $x_t/\log t = o(1)$ as $t\uparrow\infty$. 
Thus, in total,
\[\mathbb{P}\big(\xi^{(1)}(t)-\xi^{(3)}(t)\leq\lambda_t
\text{ and } \xi^{(1)}(t)> -x_t\big)
\leq \lambda_t^2 \exp\big({2x_t}(1+o(1))\big).\]
As $\e{2x_t(1+o(1))}
\leq t^\gamma$, for $t$ large enough, this concludes 
the~proof. \qedhere
\end{itemize}
\end{proof}

\begin{proof}[Proof of Proposition~\ref{prop:two_tables_summable}]
By Lemma~\ref{lem:two-tables-bound}, for any $\gamma, C>0$,
{\cec for all $k$ large enough,}
\be
\PP(\xi^{(1)}(t_k)-\xi^{(3)}(t_k)\leq \lambda_{t_k})
\leq C t_k^{2\kappa-\gamma}
= C k^{-(2\kappa\eta-\eta\gamma)}.
\ee
Since $2\kappa\eta>1$ we can choose $\gamma>0$ so that 
$2\kappa\eta-\eta\gamma>1$.  It follows that 
$\PP(\xi^{(1)}(t_k)-\xi^{(3)}(t_k)\leq \lambda_{t_k})$ are summable,
as required. 
\end{proof}

\subsection{Interpolation}\label{sub:interpolation}

By Proposition~\ref{prop:two_tables_summable}
and the Borel--Cantelli lemma,
almost surely, there exists $k_0$ such that, for all $k\geq k_0$,
\begin{equation}\label{eq:before_BC}
\Theta^{(1)}(t_k)-\Theta^{(3)}(t_k)
>  \lambda_{t_k}t_kw_{t_k}.
\end{equation}
We now show that we can `interpolate' between the times $t_k$:

\begin{proposition}\label{prop:interpolate}
As in Proposition~\ref{prop:two_tables_summable}, set 
$\lambda_t=t^{-\kappa}$ and $t_k=k^\eta$
with $\kappa,\eta>0$ satisfying $2\kappa\eta>1$.
Assume further that
\begin{itemize}[leftmargin=*]
\item In the Weibull case, that $\frac1\eta>\kappa+\frac1{1+\alpha}$,
\item In the Fr\'echet case, $\frac1\eta>\kappa+\frac1{\alpha}$,
\item In the Gumbel case, Assumption~\ref{G}.
\end{itemize}
Then, 
almost surely, there exists $k_1$ such that, for all $k\geq k_1$,
\be
\inf_{t\in[t_{k-1},t_k]} 
\big(\Theta^{(1)}(t)-\Theta^{(3)}(t)\big)
>\tfrac12 
\lambda_{t_k}t_kw_{t_k}.
\ee
\end{proposition}

To prove Proposition~\ref{prop:interpolate}, we use the following:
\begin{lemma}\label{lem:interpol}
Let $(t_k)_{k\geq 0}$ be an increasing sequence such that $t_0 = 0$.
For all $k\geq1$, for all $t\in[t_{k-1},t_k)$, 
\be
\Theta^{(1)}(t)-\Theta^{(3)}(t)\geq 
\Theta^{(1)}(t_k)-\Theta^{(3)}(t_k)-W_{m_1(t_k)}(t_k-t_{k-1})
\ee
\end{lemma}

\begin{proof}
Each $\Theta_n(t)$ is an increasing (affine) function of $t$.  Hence,
for all $i\geq1$, $\Theta^{(i)}(t)$ is increasing in $t$.  In
particular, for $t\in[t_{k-1},t_k)$, 
\ban
\Theta^{(1)}(t)-\Theta^{(3)}(t)
&\geq \Theta^{(1)}(t_{k-1})-\Theta^{(3)}(t_k)\notag\\
&=\Theta^{(1)}(t_k)-\Theta^{(3)}(t_k)
-\big[\Theta^{(1)}(t_k)-\Theta^{(1)}(t_{k-1})\big].
\ean
Because the largest exponent at time $t_k$ can only be larger than the largest exponent at time $t_{k-1}$, we have $0\leq \Theta^{(1)}(t_k)-\Theta^{(1)}(t_{k-1})$. Furthermore,
\be
0\leq \Theta^{(1)}(t_k)-\Theta^{(1)}(t_{k-1})
\leq W_{m_1(t_k)}(t_k-t_{k-1}).\label{eq:ineq_third_expo}
\ee
Indeed, the second inequality comes from the fact that
\[\Theta^{(1)}(t_k)=W_{m_1(t_k)}(t_k-\tau_{m_1(t_k)}) 
= W_{m_1(t_k)}(t_k-t_{k-1}) + W_{m_1(t_k)}(t_{k-1}-\tau_{m_1(t_k)}).\]
If $\tau_{m_1(t_k)}>t_{k-1}$, then $\Theta^{_{(1)}}_{\phantom{x}}(t_k)\leq W_{m_1(t_k)}(t_k-t_{k-1})$ and we indeed have~\eqref{eq:ineq_third_expo}.
Otherwise, $W_{m_1(t_k)}(t_{k-1}-\tau_{m_1(t_k)})$ is at most equal to the largest exponent at time~$t_{k-1}$, which is, by definition, $\Theta^{_{(1)}}_{\phantom{x}}(t_{k-1})$. This indeed implies~\eqref{eq:ineq_third_expo}.
\end{proof}

{{In the Gumbel case, we need the following facts about slowly varying functions:
\begin{lemma}\label{lem:XXX}
Let $L : (1,\infty) \to (0,\infty)$ be a non-decreasing function, slowly varying at infinity,
such that $L(x)\uparrow\infty$ as $x\uparrow\infty$,
and let $L^{-1}$ its generalised inverse.
\begin{enumerate}
\item For any $\beta>0$, 
\[\sum_{n\geq 1} \frac{n}{L^{-1}(n^\beta)}<\infty.\]
\item For any $\eps>0$, as $n\uparrow\infty$,
\[\frac{n}{L^{-1}(L(n^{1+\eps}))} \to 0, \quad\]
\end{enumerate}
\end{lemma}

\begin{proof}
{\it (1)} By the Potter bounds, for any $\delta>0$, there exists $x_0 = x_0(\delta)$ such that, for all $x\geq x_0$,
$L(x)\leq x^\delta$.
Because $L$ is non-decreasing, so is $L^{-1}$, and we get that
\[L^{-1}(L(x))\leq L^{-1}(x^\delta),\]
which implies, because $L^{-1}(L(x))\geq x$, that $x\leq  L^{-1}(x^\delta)$.
Taking $y = x^{\delta}$, we get that
$L^{-1}(y)\geq y^{\nicefrac1\delta}$. 
Taking $\delta$ large enough such that $\nicefrac\beta\delta>2$ concludes the proof.

{\it (2)} For any $K>0$, for all $n$ large enough, $L(n)^{\eps}\geq K$. Thus, because $L^{-1}$ is non-decreasing, $L^{-1}(L(n)^{1+\eps})\geq L^{-1}(KL(n))$.
By~\cite[Theorem~2.7(i)]{BGT}, $L^{-1}$ is rapidly varying, which, by definition, implies that
\[\frac{L^{-1}(L(n))}{L^{-1}(KL(n))} \to 0, \quad\text{ as }n\uparrow\infty.
\]
Thus,
\[\frac{n}{L^{-1}(L(n)^{1+\eps})} \leq \frac{L^{-1}(L(n))}{L^{-1}(L(n)^{1+\eps})} \to 0, \quad\text{ as }n\uparrow\infty,\]
which concludes the proof.
\end{proof}}}

In the Fr\'echet case 
we also need the following almost sure
estimate for the maximum weight of the $n$ first tables.
\begin{lemma}\label{lem:max_as_F}
Under Assumption~\ref{H} in the Fr\'echet case,
for any $\eps>0$, almost surely for $n$ large enough,
$\max_{1\leq i\leq n}  W_i\leq n^{2/\alpha+\eps}$. 
\end{lemma}

\begin{proof}
Using that $\mu(x,\oo)=x^{-\alpha}L(x)$,
where $L(x)$ is slowly varying at $\oo$, we get
\be
\PP\big(\max_{1\leq i\leq n}  W_i> n^{2/\alpha+\eps}\big)\leq 
n \PP(W_1\geq n^{2/\alpha+\eps})
=n^{-1-\eps\alpha}L(n^{2/\alpha+\eps}).
\ee
Then, by the Potter bounds, $\sum_{n\geq1} \PP\big(\max_{1\leq i\leq n}  W_i>
n^{2/\alpha+\eps}\big)<\oo$
so the result follows from the Borel--Cantelli lemma.
\end{proof}

\begin{proof}[Proof of Proposition~\ref{prop:interpolate}]
Note that \eqref{eq:before_BC}
combined with Lemma~\ref{lem:interpol} gives
that, for all $k\geq k_0$ and {\peter $t\in[t_{k-1},t_{k})$,}
\be\label{eq:BC-claim}
\Theta^{(1)}(t)-\Theta^{(3)}(t)
\geq \lambda_{t_k}t_kw_{t_k}
-\overline W(t_k)(t_k-t_{k-1}),
\ee
where $\overline W(t)$ is the largest table weight up to time $t$. 
We argue that
\be\label{eq:BC-claim-2}
\frac{\lambda_{t_k} t_k w_{t_k}}
{\overline W(t_k)(t_k-t_{k-1})} 
\to \oo \qquad\mbox{almost surely},
\ee
which gives the claim.
Note that
$t_k-t_{k-1} = k^{\eta}-(k-1)^{\eta}=(\eta + o(1))t_k^{1-\nicefrac1\eta}$ as $k\to\infty$.
For what follows, it is useful to refer to Table~\ref{cheatsheet} for expressions for $w_t$.
\begin{itemize}[leftmargin=*]
\item
In the Weibull case, $\overline W(t)\leq 1$ almost surely, and 
$w_t=t^{-\frac1{1+\alpha}}L_0(t)$ where $L_0(t)$ is slowly varying at infinity.  
Then, almost surely as $k\to\infty$,
$$
\frac{\lambda_{t_k} t_k w_{t_k}}
{\overline W(t_k)(t_k-t_{k-1})} \geq (\nicefrac1\eta+o(1))t_k^{\frac1\eta-\kappa-\frac1{1+\alpha}} L_0(t_k)
\to \oo,
\qquad\mbox{since } \tfrac1\eta>\kappa+\tfrac1{1+\alpha}.
$$
\item  In the Fr\'echet case, 
first note that $M(t_k)\to\oo$ almost surely,
and by large-deviations estimates for Poisson random
variables, almost surely for all $k$ large enough,
$M(t_k)\leq 2\theta t_k$.  It follows from Lemma~\ref{lem:max_as_F} that 
\be\label{eq:overlline-W-bound}
\overline W(t_k)\leq (2\theta t_k)^{2/\alpha+\eps}
\qquad\mbox{
almost surely for all $k$ large enough.}
\ee 
Also,
$w_t=t^{1/\alpha}L_3(t)$ where $L_3(t)$ is slowly varying at infinity.
Then, as $k\uparrow\infty$,
\be
\frac{\lambda_{t_k} t_k w_{t_k}}
{\overline W(t_k)(t_k-t_{k-1})} \geq (\nicefrac1\eta+o(1))
t_k^{\frac1\eta-\kappa-\frac1{\alpha}-\eps} L_3(t_k).
\ee
Since $\frac1\eta>\kappa+\frac1\alpha$ we can find $\eps>0$ such that
$\frac1\eta>\kappa+\frac1\alpha+\eps$.
Then \eqref{eq:BC-claim-2} follows.\smallskip
\item In the Gumbel case we get
\be
\frac{\lambda_{t_k} t_k w_{t_k}}
{\overline W(t_k)(t_k-t_{k-1})}
\geq (\nicefrac1\eta+o(1)) \frac{t_k^{\frac1\eta-\kappa}L_1(t_k)L_2(t_k)}
{\overline W(t_k)}
\ee
where $L_1(t)\to0$ and $L_2(t)\to M$ are both slowly varying at infinity.  
In the bounded case $M=1$, \eqref{eq:BC-claim-2} follows
for any $\kappa>0$ picking any  $\frac1\eta\in(\kappa,2\kappa)$.
In the unbounded case $M=\oo$, for any $\rho>0$, and $n\geq 1$,
\be\label{eq:maxW-G}
\PP\big(\max_{1\leq i\leq n} W_i > n^\rho\big)
\leq n \mu(n^\rho,\oo).
\ee
By Assumption~\ref{H}, we have
$A^{-1}(n^\rho) \mu(A(A^{-1}(n^\rho)),\oo) \to 1.$
{Because, in the Gumbel case, $A$ is increasing}, we get
\[\mu(n^\rho,\oo) = \frac{1+o(1)}{A^{-1}(n^\rho)}.\]
Thus, by~\eqref{eq:maxW-G},
\[\PP\big(\max_{1\leq i\leq n} W_i > n^\rho\big)\leq \frac{n(1+o(1))}{A^{-1}(n^\rho)},
\]
which, by Lemma~\ref{lem:XXX} is summable, 
{\cec because, in the Gumbel case, $A$ is slowly varying (as proved in the proof of Lemma~\ref{lem:ut})}. 
Arguing similarly to
the Fr\'echet case, we get that $\overline W(t_k)\leq M(t_k)^\rho \leq (2\theta t_k)^\rho$ almost
surely for $k$ large enough. 
Choosing $\rho>0$ such that
$\kappa<\frac1\eta+\rho<2\kappa$
\eqref{eq:BC-claim-2} follows.     \qedhere
\end{itemize}
\end{proof}

With the results obtained so far we get:

\begin{proposition}\label{prop:two-largest-exponents}
Under the same assumptions as for Proposition~\ref{prop:interpolate}, 
\be
\frac{\exp(\Theta^{(1)}(t))+\exp(\Theta^{(2)}(t))}
{\sum_{n=1}^{M(t)} \exp(\Theta_n(t))}
\to 1 \qquad \mbox{almost surely}.
\ee
\end{proposition}

\begin{proof}
Fix $t>0$. We have
\be
0\leq 
1-\frac{\exp(\Theta^{(1)}(t))+\exp(\Theta^{(2)}(t))}
{\sum_{n=1}^{M(t)} \exp(\Theta_n(t))}
\leq \frac{M(t)\e{\Theta^{(3)}(t)}}{\e{\Theta^{(1)}(t)}}.
\ee
Let $k = k(t)$ be such that $t\in[t_{k-1},t_k]$.  
Then, almost surely for all $t$ large enough, 
by Proposition~\ref{prop:interpolate}
and a large-deviations bound for $M(t_k)$,
\be\label{eq:exptheta-bound}
 \frac{M(t)\e{\Theta^{(3)}(t)}}{\e{\Theta^{(1)}(t)}}
\leq
2\theta t_k \exp\big(-\tfrac12\lambda_{t_k} t_k w_{t_k}\big).
\ee
We now check that the right-hand-side goes to zero as $t$ (and thus $k = k(t)$) goes to infinity:
\begin{itemize}[leftmargin=*]
\item In the Weibull case, 
$\lambda_{t_k} t_k w_{t_k}=t_k^{1-\kappa-\frac1{1+\alpha}}L_0(t_k)$
so \eqref{eq:exptheta-bound}
goes to zero provided we select $\kappa<1-\frac1{1+\alpha}$
and then $\eta$ as in Proposition~\ref{prop:interpolate}.
\item In the Fr\'echet case, 
$\lambda_{t_k} t_k w_{t_k}=t_k^{1-\kappa+\frac1{\alpha}}L_3(t_k)$
so \eqref{eq:exptheta-bound}
goes to zero provided we select $\kappa<1+\frac1{\alpha}$
and then $\eta$ as in Proposition~\ref{prop:interpolate}.
\smallskip
\item In the Gumbel case, 
$\lambda_{t_k} t_k w_{t_k}=t_k^{1-\kappa}L_1(t_k)L_2(t_k)$
so \eqref{eq:exptheta-bound}
goes to zero provided we select $\kappa<1$
and then $\eta$ as in Proposition~\ref{prop:interpolate}. \qedhere
\end{itemize}
\end{proof}

Proposition~\ref{prop:two-largest-exponents} can be seen as an
analog of Theorem~\ref{th:two_tables} where we have replaced the
tables with their growth rates.  
To establish 
Theorem~\ref{th:two_tables} we need to argue that
$\exp(\Theta^{(j)}(t))$ is a good approximation of the size of the
$j$-th largest table.

\subsection{Approximating the table sizes}
\label{sub:approxi}

\begin{proposition}\label{prop:exponents-vs-tables}
 {\cec Let $t_k=k^\eta$ with $\eta>0$ 
and let $\varphi\in(0,1)$.
Then, almost surely, for all $k$ large enough, 
\[
\text{for all $m$ with $\tau_m\leq t_k$,}
\sup_{t\geq \tau_m}  
|\log Z_{m}(t)-\Theta_{m}(t)|\leq t_k^{1-\varphi}.
\]}
\end{proposition}

\begin{proof}
{\cec We aim at using the Borel--Cantelli 
lemma, and thus prove that $\PP(A_k)$ is summable where
\[A_k=\Big\{\exists m \colon \tau_m\leq t_k
\text{ and }
\sup_{t\geq \tau_m}  
|\log Z_{m}(t)-\Theta_{m}(t)|> t_k^{1-\varphi}\Big\}.
\]
We have
\begin{align}
\mathbb P(A_k)
&\leq \EE\Big[
\sum_{\tau_m\leq {\cec t_k}} \PP\big( \sup_{t\geq \tau_m}  
|\log Z_{m}(t)-\Theta_{m}(t)|> t_k^{1-\varphi}\mid W_m, \tau_m\big)
\Big] \\
& = \EE\Big[
\sum_{\tau_m\leq t_k} \PP\big( 
\sup_{t\geq \tau_m}  
|\log Z_{m}(t)-W_{m}(t-\tau_m)|> t_k^{1-\varphi}\mid W_m, \tau_m\big)
\Big].\notag
\end{align}
We now use Lemma~\ref{lem:Peter} with 
$\lambda = W_m$ and $R=t_k^{1-\varphi}$
to get that, for all integers $m$ such that $\tau_m\leq t_k$,
\be\begin{split}
& \PP\big( \sup_{t\geq\tau_m} 
|\log Z_{m}(t)-W_{m}(t-\tau_m)|> {\cec t_k}^{1-\varphi}\mid W_m, \tau_m\big)
\leq 2\Gamma(\nicefrac12)\e{-t_k^{1-\varphi}/2}.
\end{split}\ee
Thus,
\[\PP(A_k)\leq 2\Gamma(\nicefrac12)\e{-t_k^{1-\varphi}/2}\EE[M(t_k)] = 
2\theta\Gamma(\nicefrac12) t_k\e{-t_k^{1-\varphi}/2}.
\]
Because $t_k = k^{\eta}$, this is summable as soon as $\varphi<1$.}
\end{proof}

\subsection{Proof of Theorem~\ref{th:two_tables}}
\label{sub:proof_main}

Recall that $m_j=m_j(t)$ denotes the index of the
$j$-th largest exponent $\Theta$.  Our first aim is to prove that 
\be\label{eq:aim}
\sup_{t\in[t_{k-1},t_k]} 
\frac{{\jeb \sum_{j=3}^{M(t)}} Z_{m_j}(t)}{Z_{m_1}(t)}
\to 0, \quad\mbox{ almost surely as }k\to\oo.
\ee
We {\peter{do this}} 
before showing how to deduce the
same claim about the largest tables, i.e.~Theorem~\ref{th:two_tables}.
{In the proof of \eqref{eq:aim}, a delicate issue is  to verify that
there  exists a choice of the parameters $\kappa, \eta$, and
$\varphi$ that satisfies all necessary assumptions. 
}

\begin{proof}[Proof of \eqref{eq:aim}]
By Proposition~\ref{prop:exponents-vs-tables}, 
almost surely for all $k$ large enough,
\begin{align}\sup_{t\in[t_{k-1},t_k]} 
\frac{{\jeb \sum_{j=3}^{M(t)} } Z_{m_j}(t)}{Z_{m_1}(t)}
&\leq \sup_{t\in[t_{k-1},t_k]}  
\frac{{\jeb \sum_{j=3}^{M(t)} } \exp(\Theta_{m_j}(t)+t_k^{1-\varphi})}
{\exp(\Theta_{m_1}(t)-t_k^{1-\varphi})}\notag\\
&\leq \sup_{t\in[t_{k-1},t_k]}  M(t)
\exp\big(
-(\Theta_{m_1}(t)-\Theta_{m_3}(t))+2t_k^{1-\varphi}
\big)\notag\\
&\leq  M(t_k)
\exp\big(-\tfrac12 t_k\lambda_{t_k} w_{t_k}+2t_k^{1-\varphi}
\big),
\label{eq:onMk}
\end{align}
by Proposition~\ref{prop:interpolate}.
Using the fact that $M(t_k)\leq 2\theta t_k$ almost surely for all $k$
large enough (by a large deviation estimate for Poisson random
variables, because $M(t_k)$ is a Poisson random variable of parameter
$\theta t_k$), we get that, almost surely for all $k$ large enough, 
\be \label{eq:onMk2}
\sup_{t\in[t_{k-1},t_k]} 
\frac{{\jeb \sum_{j=3}^{M(t)} }  Z_{m_j}(t)}{Z_{m_1}(t)}
\leq 2\theta t_k
\exp\big(-\tfrac12 t_k\lambda_{t_k} w_{t_k}+2t_k^{1-\varphi}
\big)
\ee
We need to check that the right-hand-side of \eqref{eq:onMk2} converges
to zero, i.e. that
$t_k\lambda_{t_k} w_{t_k}\gg t_k^{1-\varphi}$ as $k\uparrow\infty$.
Recall that $\lambda_t=t^{-\kappa}$.
\begin{itemize}[leftmargin=*]
\item Weibull case: $w_t=t^{-\frac1{1+\alpha}}L_0(t)$ 
so \eqref{eq:aim} follows as soon as
$-\kappa+\varphi-\tfrac1{1+\alpha}>0$ i.e.\
\be\label{eq:onMk-W}
\varphi>\kappa+\tfrac1{1+\alpha}.
\ee
\item Fr\'echet case: $w_t= t^{\frac1{\alpha}}L_3(t)$
so \eqref{eq:aim} follows as soon as
$-\kappa+\varphi+\tfrac1{\alpha}>0$ i.e.\
\be\label{eq:onMk-F}
\varphi>\kappa-\tfrac1{\alpha}.
\ee
\item
Gumbel case: $w_t$ is slowly varying
so \eqref{eq:aim} follows as soon as
\be\label{eq:onMk-G}
\varphi>\kappa.
\ee
\end{itemize}
To conclude the proof, we need to check that there exists a choice of the parameters $\kappa, \eta, \rho$, and $\varphi$ that satisfies all our assumptions.
In all cases (Weibull, Gumbel, and Fr\'echet), we have
$\lambda_t=t^{-\kappa}$ and $t_k=k^\eta$.  Our first assumption on
$\kappa$ and $\eta$ comes from Proposition
\ref{prop:two_tables_summable} and is that $2\kappa\eta>1$.  In
addition, for the three possible extreme-value distributions we have
the following assumptions:
\begin{itemize}[leftmargin=*]
\item Weibull: For Proposition~\ref{prop:interpolate}, we need
\smash{$\kappa+\tfrac1{1+\alpha}<\tfrac1\eta<2\kappa$}.
{\jeb For  Proposition~\ref{prop:exponents-vs-tables}, 
we need $\varphi\in(0,1)$,
and for \eqref{eq:onMk-W}, we need
$\varphi>\kappa+\tfrac{1}{1+\alpha}$.}
These inequalities can
only be consistent if $\alpha>1$, which is indeed the contents of
Assumption~\ref{W}.  Assuming that $\alpha>1$, we can satisfy all the
inequalities as follows.  Since $\tfrac2{1+\alpha}<1$ we can pick some
\smash{$\kappa>\tfrac1{1+\alpha}$} 
satisfying \smash{$\tfrac2{1+\alpha}<\kappa+\tfrac1{1+\alpha}<1$}. 
We can then pick {\jeb any $\varphi$ satisfying
$\kappa+\tfrac1{1+\alpha}<\varphi<1$, and any
$\eta$ satisfying $\tfrac1\eta<2\kappa$.}  
\item Fr\'echet:  For Proposition~\ref{prop:interpolate}, we need
\smash{$\kappa+\tfrac1{\alpha}<\tfrac1\eta<2\kappa$};
{\jeb 
for  Proposition~\ref{prop:exponents-vs-tables}, 
we need $\varphi\in(0,1)$;
and for \eqref{eq:onMk-F}, we need
$\varphi>\kappa-\tfrac{1}{\alpha}$.}  
These inequalities are consistent
for any $\alpha>0$, which is why we do not need a stronger assumption
in the Fr\'echet case.   To show that they can all be satisfied, we
start by picking $\kappa$ such that
$\frac1\alpha<\kappa<\frac{1+\alpha}\alpha$.  
{\jeb Note that we then have 
$\kappa+\frac1\alpha<2\kappa$ and $\kappa-\frac1\alpha<1$.  
We then
pick $\eta$ such that $\kappa+\frac1\alpha<\frac1\eta<2\kappa$
and  $\varphi$ such that
$\kappa-\frac1\alpha<\varphi<1$.}  
\item Gumbel:  Proposition~\ref{prop:interpolate} places no
  restrictions on the parameters.
{\jeb For  Proposition~\ref{prop:exponents-vs-tables}, 
we need $\varphi\in(0,1)$
and for \eqref{eq:onMk-G} we need
$\varphi>\kappa$.  In this case we simply pick any
$\kappa,\eta$ such that $0<\kappa<1$ and 
$\kappa<\frac1\eta<2\kappa$, and 
any $\varphi\in(\kappa,1)$.}
\end{itemize}

Having shown that the various inequalities can all be simultaneously 
satisfied, we conclude that the right-hand-side of
\eqref{eq:onMk} goes to 0, which means that 
\be \label{eq:interpolate}
\mathbb{P}\Big(
\frac{{\jeb \sum_{j=3}^{M(t)} }  Z_{m_j}(t)}{Z_{m_1(t)}}\to 0
 \text{ as }t\to\oo\Big)=1.
\ee
\end{proof}

Now we show how to deduce 
Theorem~\ref{th:two_tables}.

{\jeb 
\begin{proof}[{Proof of Theorem~\ref{th:two_tables}}]
Let
\be
\mathcal G=
\Big\{
\frac{\sum_{j=3}^{M(t)} Z_{m_j}(t)}{Z_{m_1}(t)}\to 0
 \text{ as }t\to\oo\Big\}.
\ee
Then \eqref{eq:interpolate} 
implies that $\PP(\mathcal G)=1$.
Let $n_i=n_i(t)$ denote the index of the $i$-th largest table.
We claim that, on $\mathcal G$,
\be
\frac{ \sum_{j=3}^{M(t)}   Z_{n_j}(t)}{Z_{n_1}(t)}\to 0
 \text{ as }t\to\oo.
\ee
First note that $n_1(t)\in\{m_1(t),m_2(t)\}$ for all large enough
$t$, since if $n_1(t)\not\in\{m_1(t),m_2(t)\}$ then
\[
\frac{\sum_{j=3}^{M(t)}  Z_{m_j}(t)}{Z_{m_1}(t)}
\geq \frac{ Z_{n_1}(t)}{Z_{n_1}(t)}= 1,
\]
which is not true on $\mathcal G$, for $t$ large enough.
Assume from now on that $n_1(t)\in\{m_1(t),m_2(t)\}$.
Consider the sets
\[
\mathcal N(t)=\{n_j(t):3\leq j\leq M(t)\}
\qquad\text{and}\qquad
\mathcal M(t)=\{m_j(t):3\leq j\leq M(t)\}.
\]
These two sets have the same size, and neither contains $n_1(t)$.  We
have two cases: either $n_2(t)\in \mathcal M(t)$ or 
$n_2(t)\not\in \mathcal M(t)$.
If $n_2(t)\in \mathcal M(t)$ then there is some $j_0\geq3$ such that 
$n_{j_0}(t)\not\in\mathcal M(t)$.
If $n_2(t)\not\in \mathcal M(t)$ then $\mathcal M(t)=\mathcal N(t)$,
in which case we set $j_0=2$.  In either case, since $j_0\geq 2$ we have
\[
\sum_{j=3}^{M(t)} Z_{m_j}(t)=\sum_{j=3}^{M(t)} Z_{n_j}(t)
-Z_{n_{j_0}}(t)+Z_{n_2}(t)\geq \sum_{j=3}^{M(t)} Z_{n_j}(t).
\]
Thus, using also that, by definition of $n_1(t)$, $Z_{n_1}(t)\geq Z_{m_1}(t)$, we get that, on $\mathcal G$,
\[
\frac{ \sum_{j=3}^{M(t)}   Z_{n_j}(t)}{Z_{n_1}(t)}
\leq
\frac{ \sum_{j=3}^{M(t)}   Z_{m_j}(t)}{Z_{m_1}(t)}\to 0,
\]
as claimed.
\end{proof}
}

\section{Further discussion}
\subsection*{Related models.}
\label{sub:other_models}
Other variants of the Chinese restaurant process perturbed by a disorder have been considered by various authors.\medskip

$\bullet$ In~\cite{MMS}, the authors discuss a model where customer
$n+1$ chooses to sit at table~$i$ with random weight~$0<W_i< 1$ with
probability $\frac1{n} S_i(n) W_i$ and occupies a new table with the
remaining probability. As in our case the  random weights are i.i.d.
If the weight distribution has no atom at~1, the authors prove that,
irrespective of the extreme value type of the weight distribution, 
the tables have microscopic occupancy and the ratio $R_n$ of the largest and second
largest table satisfies
$\lim_{n\to\infty} \PP(R_n\ge x) = \nicefrac1x$ for all $x\geq 1$.\medskip

$\bullet$ Although this does not appear in the literature (as far as we can tell), 
it would be natural, in a `weighted' Chinese restaurant process, to weigh customers instead of tables.
In this model, the $n$-th customer would have weight~$W_n$, and a new customer would join a table with probability proportional to the sum of the weights of the customers already sitting at that table, and create a new table with probability proportional to a parameter~$\theta$.
For light-tailed weight distributions at least, we expect the tables to have macroscopic occupancy in this model, just as in the classical case.
Interestingly, if $\theta = W_0$ is also a random weight, then the tables in this model can be seen as the subtrees of the root in the weighted random recursive tree, see, e.g.,~\cite{Delphin}, where this random tree is introduced and studied.
The fact that the tables in the original Chinese restaurant process can be seen as the subtrees of the root in the (non-weighted) random recursive tree is shown in~\cite{Janson19}.\medskip

$\bullet$ In the statistics literature, see, e.g.,~\cite{IL03} and the references therein, a weighted Chinese restaurant process has been studied. In this model  ``{\it customers each have a fixed affiliation and are biased to sit at tables with other customers having similar affiliations}'', see~\cite{Lindsey14}.
Affiliations can be seen as weights, and they are carried by the customers;
however, their effect on the probability to join a given table is different from 
the model described in the second bullet point just above.

\subsection*{Further results.}
\label{sub:other_results}\ \smallskip

{$\bullet$ In \cite{MSW22} an algorithm that gives access to queries about the Chinese restaurant process in sublinear time is presented. This algorithm is suitable for our model.}

\subsection*{Open problems}\label{sub:open}

{An interesting challenge is to describe 
 the length of the periods, in which the largest table remains the same 
 as a function of time. We conjecture that, for all fitness distributions $\mu$, these periods are stochastically increasing in time, a phenomenon known as \emph{ageing}. As done in \cite{MOS11} for the parabolic Anderson model, one can describe this phenomenon in the weak sense, by looking at the asymptotic probability of a change of the largest table in a given time window, and in the strong sense, by identifying an almost sure upper envelope for the process of the time remaining until the next change of profile. For the winner takes all market this corresponds to an analysis of the  slowing down in the  rate of innovation as the market expands.}


\appendix

\section{Proof of Proposition~\ref{rem:basic}}
\label{app:basic}
 
In this section we prove
Proposition~\ref{rem:basic} of the introduction. 
We use the continuous
time embedding, in which our statement becomes
\[\lim_{t\to \infty}\frac{M(t)}{\log N(t)}= \frac{\theta}{{\rm esssup }\,  \mu}.\]
Recall that from~\eqref{eq:yule} that, in continuous time,
$Z_i(t) = Y_i(W_i(t-\tau_i))$ 
where $(Y_i)_{i\geq 1}$ is a sequence of i.i.d.\ Yule processes of parameter~1,
and, for all $i\geq 1$, $\tau_i$ is the time at which table $i$ is first occupied.
Also, by~\eqref{eq:cv_yule}, almost surely as $t\uparrow\infty$
$Z_i(t)\sim \zeta_i \exp(W_i(t-\tau_i))$,
where $(\zeta_i)_{i\geq 1}$ is a sequence of i.i.d.\ standard exponential random variables.
We also recall that, by definition of the model, 
\be\label{eq:cv_M(t)}
M(t)\sim \theta t\quad\text{ almost surely as }t\uparrow\infty.
\ee

\medskip
First note that, for all $a<{\rm essup }\, \mu$,
there exists a random index $j\geq 1$ such that $W_j>a$. 
Thus, by~\eqref{eq:cv_yule}, for all $t$ large enough, $Z_j(t)\geq \exp(at)$.
Hence, by~\eqref{eq:cv_M(t)}, for all $\eps>0$, for all $t$ is large enough,  
$M(t)/\log N(t) \leq  {(1+\eps)\theta}/{a}$.
If ${\rm essup } \, \mu = \oo$, this concludes the proof, since one can make $a$ arbitrarily large and conclude that $M(t)/\log N(t)\to 0$ almost surely as $t\uparrow\infty$, as claimed.
In the case when $a:={\rm essup}\,  \mu<\infty$,
note that, by~\eqref{eq:cv_yule}, for all $t$ large enough, $N(t)\leq 2\Xi_t \exp(at)$,
where $\Xi_t$ is the sum of $M(t)$ independent standard exponentials. 
Hence, for all $\eps>0$, for all sufficiently large $t$, $\log N(t) \leq  (1+\eps)at$ and $M(t)\geq (1-\eps)t$, by~\eqref{eq:cv_M(t)}, which implies $M(t)/\log N(t) \geq \frac{(1-\eps)\theta}{(1+\eps)a}$. 
Since $\eps>0$ is arbitrary, this implies~(i).\smallskip


Now fix a table number $i\in\mathbb N$.  Recall that, by~\eqref{eq:cv_yule}, 
$Z_i(t) \sim \zeta_i \exp(W_i(t-\tau_i))$,
which clearly implies that $Z_i(t)\to\infty$ 
as $t\uparrow\infty$, because $\tau_i<\infty$ almost surely.
Because $\mu$ has no atom at its essential supremum, 
there exists almost surely a random index 
$j\neq i$ such that $W_j>W_i$.
Using~\eqref{eq:cv_yule} again, we get that $Z_i(t)/Z_j(t)\to
0$ as $t\uparrow\infty$ almost surely. 
If $N(t)$ denotes the number of customers in the restaurant at time $t$, 
then $Z_i(t)/N(t)\leq Z_i(t)/Z_j(t)\to 0$ as $t\uparrow\oo$ almost surely, 
so that table~$i$ cannot have macroscopic occupancy, as claimed in (ii) and~(iii). \smallskip

{
To see (iv), assume that the proportion of customers at the largest table converges almost surely to one. 
On this event, there exists $N>4$ such that 
\[\max_{i\geq 1} \frac{S_i(n)}{n}>\frac34 \text{ for all }n\geq N.\]
Let $i_N$ denote the index of the unique largest table at time $N$: 
the function $f(n):=S_{i_N}(n)/n$ takes a value larger than $3/4$ at $n=N$ and, 
by (iii), it goes to zero
as $n\to\infty$. 
Note that, for all $m\geq N$,
$|f(m+1)-f(m)| \leq \tfrac1N$
and hence there exists some $M\geq N$ such that
\[|f(M)-\tfrac12| \leq \tfrac1N.\]
Hence, $i_N$ is not the index of the largest table at time $M$,
and for the index $i_M$ of the largest table at time $M$, 
$S_{i_M}(M)/M \leq (M-S_{i_N}(M))/M \leq 1/2+\frac1N$, contradicting our assumption.}

\section{Examples of weight distributions}
\label{app:examples}

\subsection{Examples satisfying Assumption~\ref{H}}

We give four examples of probability distributions $\mu$ that
satisfy Assumption~\ref{H}; for each of these, we give formulas for
$A(t)$, $B(t)$, $u_t$, $v_t$ and $w_t$.\smallskip 

\begin{example}[Weibull]
For $\alpha>0$ let $\mu(1-x, 1) = x^{\alpha}$ for all $x\in[0,1]$. Then,
for all $x\geq 0$,
\[t\mu(1-xt^{-\nicefrac1\alpha}, 1) = x^\alpha,\]
and thus Assumption~\ref{H} is satisfied with $A(t) = 1$, $B(t) =
t^{-\nicefrac1{\alpha}}$ and $\Phi(x) = |x|^{\alpha}$ for all $x\leq 0$,
and $\Phi(x) = 0$ otherwise.
We get from \eqref{eq:cond_u} that
\be
u_t=t^{\frac{\alpha}{1+\alpha}},\quad
v_t=1,\quad
w_t=t^{-\frac1{1+\alpha}}.
\ee
Since there is equality in~\ref{eq:maxdomain-2}, the convergence in $L^1$ of~\eqref{eq:maxdomain-3} holds straightforwardly.
\end{example}

\begin{example}[Gumbel bounded] \label{ex:Gumbel}
For $\alpha>0$ let $\mu(1-x,1) =
\exp(1-x^{-\alpha})$ for all $x\in[0,1]$. Then, for all $x\in\mathbb R$,
\[
t\mu(1-(1+\log t)^{-\frac1\alpha}+x(1+\log t)^{-\frac1\alpha-1}/\alpha, 1) 
\to\mathrm{e}^{1-\alpha x}.
\]
Thus, Assumption~\ref{H} is satisfied with $A(t) = 1-(1+\log
t)^{-\frac1\alpha}$, $B(t) = \tfrac1\alpha(1+\log t)^{-\frac1\alpha-1}$ and $\Phi(x) =
\mathrm e^{-x}$ for all $x\in\mathbb R$.  
We identify $u_t$ as in the proof of Lemma~\ref{lem:ut}, namely
$u_t=f^{-1}(t)$ where  
\[f(u)=uA(u)/B(u)=u(\log u)((\log u)^{1/\alpha}-1).\]
This implies that $u_t = t(\log t)^{-\frac{\alpha+1}{\alpha}}(\nicefrac1\alpha+o(1))$, 
and thus $v_t =1-(\log t-(1+\nicefrac1\alpha)\log\log t)^{-\frac1\alpha}(1+o(1))$, and 
$w_t = (\log t)^{-\frac{\alpha+1}\alpha}(\nicefrac1\alpha+o(1))$. 
We now check that~\eqref{eq:maxdomain-3} holds: for all $x>0$,
\ba
 t\mu(A(t) &\,+uB(t), 1)\mathrm du\\
&= \int_x^{1+\log t} t \exp\big(1-((1+\log t)^{-\nicefrac1\alpha} -\tfrac u\alpha(1+\log t)^{-1-\nicefrac1\alpha})^{-\alpha}\big)\mathrm du\\
&= \int_x^{1+\log t} t \exp\big(1-(1+\log t)\big(1 -\tfrac u{\alpha(1+\log t)}\big)^{-\alpha}\big)\mathrm du.
\ea
To use the dominated convergence theorem note that, for all $x\leq u\leq 1+\log t$,
\[0\leq t \exp\big(1-(1+\log t)\big(1 -\tfrac u{\alpha(1+\log t)}\big)^{-\alpha}\big)
\leq t \exp\big(1-(1+\log t)\big(1 +\tfrac u{1+\log t}\big)\big) = \mathrm e^{-u},\]
because, for all $w\in(0,1)$, $(1-w)^{-\alpha}\geq 1+\alpha w$.
As $u\mapsto \mathrm e^{-u}$ is integrable on $[x,\infty)$,
the dominated convergence theorem applies and we can conclude that~\eqref{eq:maxdomain-3} holds.
\end{example}

\begin{example}[Gumbel unbounded] 
For $\alpha>0$ let
$\mu(x,\infty) = \exp(-x^{\alpha})$ for all $x\geq 0$. 
Then
\[
t\mu((\log t)^{\frac1\alpha}+x(\log t)^{\frac1\alpha-1}/\alpha, \infty) 
\to \mathrm{e}^{-x}.
\]
Thus, Assumption~\ref{H} is satisfied with $A(t) = (\log
t)^{\frac1\alpha}$, $B(t) = \tfrac1\alpha(\log t)^{\frac1\alpha-1}$ and $\Phi(x) =
\mathrm e^{-x}$ for all $x\in\mathbb R$.
Similarly to before we have $u_t=f^{-1}(t)$ where this time
 $f(u)=u(\log u)$. 
This implies that $u_t = (1+o(1))t/\log t$, and 
thus $v_t = (\log t)^{\nicefrac1\alpha}-(\log\log t) (\log t)^{\nicefrac1\alpha-1}(\nicefrac1\alpha+o(1))$, and
$w_t \sim \tfrac1\alpha(\log t)^{\frac1\alpha-1}$.
Checking~\eqref{eq:maxdomain-3} is similar to Example~\ref{ex:Gumbel}.
\end{example}

\begin{example}[Fr\'echet]
 For $\alpha>0$ let  $\mu(x,\infty) =
x^{-\alpha}$ for all $x\geq 1$. Then
\[t\mu(xt^{\nicefrac1{\alpha}}, \infty) = x^{-\alpha},\]
and thus Assumption~\ref{H} is satisfied with $A(t) = 0$, 
$B(t) = t^{\nicefrac1\alpha}$, 
and $\Phi(x) = x^{-\alpha}$ for all $x> 0$, and $\Phi(x) = \infty$ for all $x\leq 0$.
As discussed, in this case we take $v_t=0$ and we take 
$u_t = t$ instead of taking it as a
solution of~\eqref{eq:cond_u}.
We get that $w_t = B(t)= t^{\nicefrac1\alpha}$.
\end{example}


\subsection{Examples satisfying Assumption~\ref{G}}

We list a few examples satisfying Assumption~\ref{G}.
When $M=1$ we write
$\mu(x, 1) = \exp(-m(x))$ for all $x\in [0,1)$. Then
the  following  weight
distributions, given by a suitable  function $m$,  all satisfy
Assumption~\ref{G}.
\begin{multicols}{2}
\begin{enumerate}
\itemsep4pt
\item[(a)] $m(x) = (1-x)^{-\alpha}-1$ for $\alpha>0$;
\item[(b)] $m(x) = \mathrm e^{\frac1{1-x}}-\mathrm e$;
\item[(c)] $m(x) = \frac x{1-x}$;
\item[(d)] $m(x) = \mathrm e^{\frac1{\sqrt{1-x}}}-\mathrm e$;
\item[(e)] $m(x) = \mathrm{tan}(\pi x/2)$;
\end{enumerate}
\end{multicols}
\noindent
Here (a--e) also satisfy von Mises' condition~\cite[Proposition~1.1(b)]{Resnick},
which is a sufficient condition for $\mu$ to belong to the domain of
attraction of the Gumbel distribution. 
{Note that we are unable to prove that Assumption~\ref{G} 
is satisfied by all distributions that satisfy the von Mises condition. 
We are also unable to provide an example of weight distribution that belongs to the domain of attraction of the Gumbel distribution, 
satisfies Assumption~\ref{G}, and does not satisfies the von Mises condition.
However, the function $m(x) = \log\big(\tfrac{\mathrm e}{1-x}\big)\log\log\big(\tfrac{\mathrm e}{1-x}\big)$, for all $x\in[0,1)$, corresponds to a weight distribution $\mu$ that is in the domain of attraction of the Gumbel distribution and does not satisfy Assumption~\ref{G} (this distribution does not satisfy the von Mises condition).
Examples (a--e) are all bounded weight distributions.
The following is an
unbounded example:
\begin{itemize}
\item[(f)]
$\mu(x,\infty) = \exp(-x^\alpha)$ for any $\alpha>1$.
\end{itemize}

We prove next that
(a) satisfies Assumption~\ref{G}.  The others are similar.
Recall  that, in this example, 
$\mu(1-x, 1) = \exp(1-x^{-\alpha})$ for some $\alpha>0$
and all $x\in(0,1]$.
Assumption~\ref{H} is satisfied with
\[
A(t) = 1 - (1+\log t)^{-\frac1\alpha} 
\quad\text{ and }\quad 
B(t) = \tfrac1{\alpha}(1+\log t)^{-\frac{\alpha+1}{\alpha}}.
\]
We also set 
$\hat A(t) = 1-A(t)=(1+\log t)^{-\frac1\alpha}$.
For all $t\geq 0$ and  all $x\in\mathbb R$, we have
\[
t\mu(A(t)+xB(t), 1) =t\exp\bigg[1-\hat A(t)^{-\alpha}
\big(1-\tfrac{xB(t)}{\hat A(t)}\big)^{-\alpha}\bigg].
\]
Now note that, for all $y<1$, $(1-y)^{-\alpha}\geq 1 + \alpha y$.
Thus, for all $x< \hat A(t)/B(t)=\alpha(1+\log t)$, we have 
\[
t\mu(A(t)+xB(t), 1) 
\leq t\exp\bigg[1-\hat A(t)^{-\alpha} \big(1+\alpha\tfrac{xB(t)}{\hat A(t)}\big)\bigg]
=t\exp\big(1-(1+\log t) -  x\big)= \mathrm e^{-x}.
\]
Making the change of variables $t\mapsto u_t$, this says:
\[
\mbox{if } x\in(-\oo,\alpha(1+\log t)) \mbox{ then }
\Phi_t(x)\leq \e{-x},
\]
which establishes the 
upper bound in Assumption~\ref{G}\ref{G2}.
For the lower bound, note that there exists a constant $C>0$ such
that, for all $y\in [-1, \nicefrac12]$,
$(1-y)^{-\alpha}\leq 1+\alpha y + Cy^2$.  Therefore,
for all $x\in (-\hat A(t)/B(t), \hat A(t)/2B(t))$ we have
\[\begin{split}
t\mu(A(t)+xB(t), 1) &\geq
t\exp\bigg[1-\hat A(t)^{-\alpha} 
\big(1+\tfrac{\alpha xB(t)}{\hat A(t)}+C \tfrac{x^2B(t)^2}{\hat A(t)^2}
\big)\bigg]\\
&= \exp\Big(-x-C\tfrac{x^2B(t)^2}{\hat A(t)^{2+\alpha}}\Big).
\end{split}\]
Note that
\[
\frac{B(t)^2}{\hat A(t)^{2+\alpha}} = \frac{1}{\alpha^2(1+\log t)},
\]
thus after the change of variables $t\mapsto u_t$ we have
\[
\mbox{if } x\in(-\alpha(1+\log t),\tfrac12\alpha(1+\log t)) \mbox{ then }
\Phi_t(x)\geq \e{-x}\exp\big(-x^2\tfrac{C}{\alpha^2(1+\log t)}\big)
\]
which concludes the proof of the lower bound in 
Assumption~\ref{G}\ref{G2}.
\pagebreak[3]\smallskip

For Assumption~\ref{G}\ref{G3},
recall that $u_t$ is defined as the unique solution of
\[
\alpha u_t \big(1 - (1+\log u_t)^{-\frac1\alpha}\big) = 
t (1+\log u_t)^{-\frac{\alpha}{\alpha+1}}.
\]
Hence $\log u_t \sim \log t$ as $t\uparrow\infty$ and $u_t = t\hat u_t$ with $\log \hat u_t = o(\log t)$. Thus, 
$\alpha\hat u_t  \sim (\log t)^{-\frac{\alpha}{\alpha+1}}$ and so
$\hat u_t \sim \frac1\alpha (\log t)^{-\frac{\alpha}{\alpha+1}}$.
This implies
\[
u_t = (\nicefrac1\alpha+o(1))t(\log t)^{-\frac{\alpha}{\alpha+1}}.
\]
Therefore
\[
L_1(t)=u_t/t=\frac{1/\alpha+o(1)}{(\log t)^{\frac\alpha{1+\alpha}}}
\]
so clearly $L_1(t)\log\log t\to 0$.

\section{A large deviations bound for the Yule process}

{\jeb 
\begin{lemma}\label{lem:Peter}
Let $(Y_t \colon t\geq 0)$ be a Yule process with
parameter~$\lambda>0$ and let $R>0$.
Then,
\[\mathbb P\big( 
\sup_{t\geq 0}
\big| \log Y_t-\lambda t\big| \geq R \big) 
\leq 2\Gamma(\nicefrac12)\e{-R/2}.\]
\end{lemma}

\begin{proof}
First note that, for any $T\geq 0$,
\begin{align}
\mathbb P\big( 
\sup_{t\in[0,T]}
\big| \log Y_t-\lambda t\big| \geq R \big) &\leq
\PP\big( \sup_{t\in[0,T]}
 \log Y_t-\lambda t \geq R \big) +
\PP\big( \inf_{t\in[0,T]}
\log Y_t-\lambda t \leq -R \big) \notag\\
&=
\PP\big( \sup_{t\in[0,T]}
 \frac{Y_t}{\e{\lambda t}} \geq \e{R} \big) +
\PP\big( \inf_{t\in[0,T]}
 \frac{Y_t}{\e{\lambda t}} \leq \e{-R} \big).
 \label{first decomposition}
\end{align}
Now, $(Y_t/\e{\lambda t})_{t\geq 0}$ is a martingale started at~1.
Thus 
by Doob's maximal inequality,
and using 
$\mathbb E[Y_T/\e{\lambda T}]=1$, we have
\be\label{eq:borne1}
\PP\big( \sup_{t\in[0,T]}
 \frac{Y_t}{\e{\lambda t}} \geq \e{R} \big)
\leq \frac{\EE[\frac{Y_T}{\e{\lambda T}}] }{\e{R}}=\e{-R}.
\ee
On the other hand, {\cec for any $0<\eps<1$,}
\[
\PP\big( \inf_{t\in[0,T]}
 \frac{Y_t}{\e{\lambda t}} \leq \e{-R} \big)
=\PP\big( \sup_{t\in[0,T]}
 \big(\frac{Y_t}{\e{\lambda t}} \big)^{-\eps}\geq \e{\eps R} \big).
\]
Since $x\mapsto x^{-\eps}$ is convex, 
$\big(\frac{Y_t}{\e{\lambda t}} \big)^{-\eps}$ is a submartingale.  Thus, 
by Doob's maximal inequality again,
\[
\PP\big( \sup_{t\in[0,T]}
 \big(\frac{Y_t}{\e{\lambda t}} \big)^{-\eps}\geq \e{\eps R} \big)
\leq
\frac{\EE[\big(\frac{Y_T}{\e{\lambda T}}\big)^{-\eps}] }{\e{\eps R}}
=\e{-\eps R+\eps\lambda T}\EE[Y_T^{-\eps}].
\]
{\cec
To finish the proof we recall \cite[Section~III.5]{AthreyaNey}
that $Y_T$ has the geometric distribution with parameter
$p=\e{-\lambda T}$.  By Proposition \ref{prop:Cecile} below,
\[
\EE[Y_T^{-\eps}]\leq 
\frac{\e{-\eps\lambda T}}{1-\e{-\lambda T}}\Gamma(1-\eps).
\]
Thus, 
\[
\PP\big( \inf_{t\in[0,T]}
 \frac{Y_t}{\e{\lambda t}} \leq \e{-R} \big)
\leq \frac{\e{-\eps R}}{1-\e{-\lambda T}}\Gamma(1-\eps)
\]
Taking $\eps = \nicefrac12$ and combining with~\eqref{first decomposition}
and~\eqref{eq:borne1}, we get that
\[\mathbb P\big( 
\sup_{t\in[0,T]}
\big| \log Y_t-\lambda t\big| \geq R \big)
\leq  \frac{2\Gamma(\nicefrac12)\e{- R/2}}{1-\e{-\lambda T}}.
\]
As the event on the left is increasing in $T$, 
letting $T\uparrow\infty$ concludes the proof.
}
\end{proof}
}

{\cec 
\begin{proposition}\label{prop:Cecile}
Let $Y$ be geometrically distributed with parameter $p\in[0,1]$ and
let $\eps\in(0,1)$.  Then 
\[
\EE[Y^{-\eps}]\leq \frac{p^\eps}{1-p}\Gamma(1-\eps).
\]
\end{proposition}
\begin{proof}
Using the change of variables $u=x\log(1/(1-p))$:
\[\begin{split}
\EE[Y^{-\eps}]&=\sum_{k=1}^\oo \frac{p(1-p)^{k-1}}{k^\eps}
\leq \frac{p}{1-p}\sum_{k=1}^\oo 
\int_{k-1}^k \dd x \; \frac{(1-p)^{x}}{x^\eps}\\
&=\frac{p}{1-p}\int_0^\oo\dd x\; x^{-\eps} \exp(-x\log(1/(1-p)))\\
&=\frac{p}{1-p}(\log(\tfrac1{1-p}))^{\eps-1}
\int_0^\oo \dd u\; u^{-\eps}\e{-u}
=\frac{p}{1-p}(\log(\tfrac1{1-p}))^{\eps-1}\Gamma(1-\eps)\\
&\leq \frac{p^\eps}{1-p}\Gamma(1-\eps),
\end{split}\]
where the last step used $1-p\leq\e{-p}$.
\end{proof}
}

\bibliographystyle{alpha}
\bibliography{ICRP}
%
%
%
%
%
%

\end{document}